\documentclass[a4paper,reqno,oneside]{amsart}

\usepackage{a4wide}
\usepackage{hyperref}
\usepackage{amsmath,amsfonts,amssymb,amsthm}
\usepackage[utf8]{inputenc}
\usepackage{ifthen}
\usepackage[style=american]{csquotes}
\usepackage{color}
\usepackage{graphicx}
\usepackage{todonotes}

\usepackage{enumitem}

\newtheorem{theorem}{Theorem}[section]
\newtheorem{lemma}[theorem]{Lemma}

\newtheorem{proposition}[theorem]{Proposition}

\newcommand{\Ex}{\mathbb E}
\newcommand{\cA}{\mathcal{A}}
\newcommand{\cB}{\mathcal{B}}
\newcommand{\cC}{\mathcal{C}}
\newcommand{\cD}{\mathcal{D}}

\newcommand{\cP}{\mathcal{P}}

\newcommand{\cS}{\mathcal{S}}
\newcommand{\cT}{\mathcal{T}}
\newcommand{\cbT}{\mathcal{\overline{T}}}

\newcommand{\bD}{\overline{D}}
\newcommand{\bP}{\overline{P}}
\newcommand{\bS}{\overline{S}}
\newcommand{\brho}{\overline{\rho}}

\newcommand{\tD}{\widetilde{D}}
\newcommand{\tP}{\widetilde{P}}
\newcommand{\tS}{\widetilde{S}}

\newcommand{\hD}{\hat{D}}
\newcommand{\hP}{\hat{P}}
\newcommand{\hS}{\hat{S}}

\newcommand{\cbD}{\mathcal{\overline{D}}}
\newcommand{\cbS}{\mathcal{\overline{S}}}
\newcommand{\cbP}{\mathcal{\overline{P}}}

\newcommand{\ctD}{\mathcal{\widetilde{D}}}
\newcommand{\ctS}{\mathcal{\widetilde{S}}}
\newcommand{\ctP}{\mathcal{\widetilde{P}}}

\newcommand{\chD}{\mathcal{\hat{D}}}
\newcommand{\chS}{\mathcal{\hat{S}}}
\newcommand{\chP}{\mathcal{\hat{P}}}

\newcommand{\bigO}{\mathcal O}

\DeclareMathOperator{\Arg}{Arg}

\DeclareMathOperator{\De}{\Delta e}
\newcommand{\cTs}{\mathcal{T}^s}

\begin{document}
\title[Spanning trees in random series-parallel graphs]{Spanning trees in random series-parallel graphs}

\author{Julia Ehrenm\"uller}
\address{(JE) Technische Universitat Hamburg-Harburg, Institut f\"ur Mathematik, Am Schwarzenberg-Campus 3, 21073 Hamburg, Germany }
\email{julia.ehrenmueller@tuhh.de}

\author{Juanjo Ru\'e}
\address{(JR) Freie Universit\"at Berlin, Institut f\"ur Mathematik und Informatik, Arnimallee 3, 14195 Berlin, Germany}
\email{jrue@zedat.fu-berlin.de}

\thanks{A preliminary version of the results of this paper was presented at the \emph{Bordeaux Graph Workshop} held in Bordeaux in November 2014.
J.\,R.~was partially supported by the FP7-PEOPLE-2013-CIG project CountGraph (ref. 630749), the Spanish MICINN projects MTM2014-54745-P and MTM2014-56350-P, the DFG within the Research Training Group \emph{Methods for Discrete Structures} (ref. GRK1408), and the \emph{Berlin Mathematical School}.
}




\begin{abstract}
By means of analytic techniques we show that the expected number of spanning trees in a connected labelled series-parallel graph on $n$ vertices chosen uniformly at random satisfies an estimate of the form
$$s \varrho^{-n} (1+o(1)),$$
where $s$ and $\varrho$ are computable constants, the values of which are approximately $s \approx 0.09063$ and $\varrho^{-1} \approx 2.08415$.
We obtain analogue results for subfamilies of series-parallel graphs including 2-connected series-parallel graphs, 2-trees, and series-parallel graphs with fixed excess.
\end{abstract}
\maketitle



\section{Introduction}

The study of spanning trees and their enumeration is a central question in graph theory and in combinatorial optimization.
It is well known that the number of spanning trees of a given graph $G$ is an evaluation of its Tutte polynomial (see for instance \cite{bi93}).
A lot of research has been devoted to study estimates of this number when dealing with restricted graph families.
For instance, various results have been obtained for regular graphs and for graphs with degree constraints (see e.g.~\cite{Alon90, Kost95, Ly05, McKay83}).

The enumeration of graphs with a distinguished spanning tree has also been studied extensively in the context of planar maps (namely, embedded connected planar graphs in the sphere, see Schaeffer's Chapter at \cite{handbook15} for an introduction to this area).
The first result in this context was obtained in the sixties by Mullin who was studying the number of rooted planar maps on $n$ edges with a distinguished spanning tree \cite{Mullin67}.
Mullin determined that this number is $C_n C_{n+1}$, where $C_n$ stands for the $n$-th Catalan number.
Such formula was explained later by Cori, Dulucq and Viennot by means of Baxter permutations~\cite{CoDuVi86},
and by Bernardi using a direct bijection with pairs of plane trees with $n$ and $n+1$ edges, respectively (see \cite{Bern07}).
Recently, Bousquet-M\'elou and Courtiel investigated the enumeration of regular planar maps carrying a distinguished spanning \emph{forest}, as well as the connections of these counting formulas with statistical models as the Potts model~\cite{BousquetMelou15} (see also \cite{Bernardi08} for the connection of spanning trees on maps and the Tutte polynomial).

In this paper we study spanning trees in series-parallel graphs. A graph is \emph{series-parallel} (or SP for short) if it is $K_4$-minor free. Over the past few decades, SP graphs have been extensively studied from various points of view both in graph theory and computer science.
In particular, being a subclass of planar graphs and a superclass of outerplanar graphs, SP graphs turned out to serve well as a pre-stage for the analysis and study of problems on planar graphs.
Indeed, the family of SP graphs is the prototype of the so-called \emph{subcritical} graph class family (see e.g.~\cite{Drmota2013, GiNoyRue2013}).
In a typical connected graph in such a family, maximal 2-connected subgraphs (also called \emph{blocks}) are small compared with the total size of the graph.
This behaviour arises as a consequence of a subcritical composition phenomenon which appears in the specification of the generating functions associated with connected graphs of the family.

In this paper we focus on enumerative problems defined on SP graphs.
To this purpose, let us quickly state the following alternative definition of SP graphs that gives more insight into their structure and also justifies their name.
Let $G$ be a graph and let $s$ and $t$ be two of its vertices.
We say $G$ is \emph{series-parallel with terminals $s$ and~$t$}
if $G$ can be turned into the single edge $\{s,t\}$ by a sequence of the following operations:
replacement of a pair of \emph{parallel edges} (i.e.~edges that have two common endpoints) by a single edge,
or replacement of a pair of \emph{series edges} (i.e.~non-parallel edges that have a common endpoint of degree 2) by a single edge.
A graph~$G$ is \emph{2-terminal series-parallel} if there are vertices $s$ and~$t$ in~$G$ such that~$G$ is series-parallel with terminals~$s$ and~$t$.
Finally, a graph $G$ is series-parallel if and only if each of its 2-connected components is a 2-terminal series-parallel graph (see e.g.~\cite{BrandstadtLS99}).

Also, SP graphs are known to be the class of graphs of treewidth at most 2 (see e.g.~\cite{BrandstadtLS99}).
Edge-maximal SP graphs (i.e.~graphs which cease to be SP whenever any missing edge is added) are exactly the class of 2-trees, which can be defined in the following way.
A single edge is a 2-tree. If $T$ is not a single edge,
then $T$ is a 2-tree if and only if there exists a vertex $v$ of degree~2 such that its neighbours are adjacent and $T-v$ is also a 2-tree.
Conversely, every subgraph of a 2-tree is a SP graph. In particular, SP graphs are at most 2-connected since 2-connected SP graphs always contain a vertex of degree 2.

From now on, unless stated otherwise, all graphs under study are labelled and simple.
By a random object of a given family we mean an object chosen uniformly at random from all the elements of the same size, e.g.~graphs on the same number of vertices. In the present paper we study enumerative properties of spanning trees and spanning forests on random SP graphs. Before stating our results, let us survey some relevant related investigations.

One can easily verify that the number of edges of an $n$-vertex 2-tree is precisely $2n-3$.
In the same context, Moon~\cite{Moon1970} showed that the number of 2-trees on $n$ vertices is equal to ${n\choose 2}(2n-3)^{n-4}$.
The enumeration of SP graphs is, however, more involved.
Bodirsky, Gim\'enez, Kang, and Noy proved in \cite{Bodirsky2007} that the number of connected SP graphs on $n$ vertices is asymptotically of the form
\begin{equation*}
\label{eq:numberSP}
c_{s} n^{-5/2}\varrho_{s}^{-n} n!,
\end{equation*}
where $c_s\approx 0.00679$ and $\varrho_s \approx 0.11021$ are computable constants. In the same paper they showed that the number of edges in a random connected SP graph is asymptotically normally distributed with mean asymptotically equal to $\kappa n $ and variance asymptotically equal to $\lambda n$, where $\kappa \approx 1.61673$ and $\lambda \approx 0.2112$ are again computable constants.

Building on these results, a lot of research has been done to understand the qualitative picture that emerges when studying a random SP graph with a fixed number of vertices. The maximum degree and the degree sequence of a random SP graph have been studied in \cite{Drmota2011} and \cite{Bernasconi2009, Drmota20102}, respectively. Drmota and Noy~\cite{Drmota2013} investigated several extremal parameters in subcritical graph classes, which include the class of SP graphs. They showed, for instance, that the expected diameter $D_n$ of a random connected SP graph on $n$ vertices satisfies $c_1 \sqrt{n} \leq \Ex[D_n] \leq c_2 \sqrt{n \log n}$ for some positive integers $c_1$ and $c_2$. The precise asymptotic estimate has been proved very recently by Panagiotou, Stufler, and Weller~\cite{panagiotou2015} to be of order $\Theta(\sqrt{n})$. In the same work, the authors exploited this fact to prove that in subcritical graph classes, and in particular in SP graphs, the normalized metric space $(V(G), d_{G}/\sqrt{n})$ (where $d_G(u, v)$ is the number of edges in a shortest path that contains $u$ and $v$ in $G$) converges with respect to the Gromov-Hausdorff metric to the \emph{Brownian Continuum Random Tree} multiplied by a constant scaling factor that depends on the class under study (see~\cite{panagiotou2015}).

\subsection*{Our results.}
In the present paper we study the number of spanning trees in a random (connected or 2-connected) SP graph on $n$ vertices. In particular, our main result is a precise estimate of the expected number of spanning trees.

\begin{theorem}\label{thm:main}
Let $X_n$ and $Z_n$ denote the number of spanning trees in a connected, respectively 2-connected, labelled SP graph on $n$ vertices chosen uniformly at random. Then,
\begin{eqnarray*}
\mathbb{E}[X_n]=&s \varrho^{-n} (1+o(1)),& \text{where } s \approx 0.09063,\,\,\, \varrho^{-1} \approx 2.08415,\\
\mathbb{E}[Z_n]=&p \varpi^{-n}(1+o(1)),& \text{where } p \approx 0.25975 ,\,\,\, \varpi^{-1}\approx 2.25829.\nonumber
\end{eqnarray*}
\end{theorem}

The previous analysis is done over \emph{all} (connected, 2-connected) SP on a given number of vertices. However, we also address the study of extremal situations. First, we can also particularize the computation of the expectation in the case of a random 2-tree on $n$ vertices, which maximizes the number of edges in an $n$-vertex SP graph. In this case, the expected value of the number of spanning trees is slightly bigger than the one in Theorem~\ref{thm:main}.

\begin{theorem}
\label{thm:2trees}
Let $U_n$ denote the number of spanning trees in a labelled 2-tree on $n$ vertices chosen uniformly at random. Then, the expected value of $U_n$ is asymptotically equal to $s_2 \varrho_2^{-n}$, where $s_2 \approx 0.14307 $ and $\varrho_2^{-1} \approx 2.55561$.
\end{theorem}

Finally, we study SP graphs with few edges.
More precisely, we elaborate the expected number of spanning trees in a random connected SP graph on $n$ vertices and \emph{fixed excess} equal to $k$, where $k$ is a integer that does not depend on $n$.
Recall that the excess of a graph $G$ is defined as the number of its edges minus the number of its vertices (by fixed we mean that it does not grow as a function of $n$).
Our result is a polynomial estimate (in $n$) of the expected number of spanning trees:

\begin{theorem}
\label{thm:excess}
Let $k>1$ be a fixed integer. Let $X_{n,k}$ denote the number of spanning trees in a connected labelled SP graph, on $n$ vertices and with fixed excess equal to $k$, chosen uniformly at random.

Then, when $n$ is large enough,
$$\mathbb{E}[X_{n,k}]= \tilde{c}(k) \frac{\Gamma(3k/2)}{\Gamma(2k+1/2)} \left(\frac{n}{2}\right)^{\frac{k+1}{2}}  (1+o(1)),$$
where the function $\tilde{c}(k)$ satisfies that, for $k$ large
\begin{equation}\label{eq:cons_k}
\tilde{c}(k)= \tilde{c}   \tilde{\gamma}^{-k} (1+o(1)),
\end{equation}
with $\tilde{c} \approx 0.90959$ and $\tilde{\gamma}^{-1} \approx 2.60560$.
\end{theorem}

The previous formulas must be understood in the following way: we fix $k$ and we let $n$ tend to infinity.
Additionally, if $k$ is sufficiently large, we can get the approximation of $\tilde{c}(k)$ stated in the second part of Theorem~\ref{thm:excess}.
Indeed, the term $o(1)$ in Equation \eqref{eq:cons_k} is polynomially small in $k$ ($O(k^{-1})$).

In order to deduce these expressions in Theorem~\ref{thm:excess} we analyse weighted cubic SP multigraphs on $2k$ vertices. These objects are reminiscent of the work \cite{Kang2012} and are building on previous enumerative results on simple cubic planar graphs~\cite{BodirskyKang2007}, see Section \ref{sec:fixedexcess} for definitions and details. The asymptotic estimate stated in Theorem \ref{thm:excess} arises when getting asymptotic estimates (in terms of $k$) for the number of such multigraphs.

\subsection*{Organisation.} The rest of the paper is organised as follows. In Section~\ref{sec:prel} we introduce the essential combinatorial and in Section~\ref{sec:analytic} the essential analytic definitions, techniques, and results that we use. The proof of Theorem~\ref{thm:main} is then presented in Subsection~\ref{sec:1stmoment} of Section~\ref{sec:trees}.
In the same section we also analyse the behaviour of the growth constant of the expected number of spanning trees in a random connected SP graph if we fix its edge density (Subsection~\ref{sub:fix-y}) and comment on the variance of the number of spanning trees in a random 2-connected SP graph (Subsection~\ref{sec:2ndmoment}). Next, Section~\ref{sec:2trees} is devoted to the analysis of 2-trees and the proof of Theorem~\ref{thm:2trees}. Then, Section~\ref{sec:fixedexcess} deals with SP graphs with fixed excess and presents the proof of Theorem~\ref{thm:excess}. Finally, Section~\ref{sec:concluding} contains some concluding remarks and open problems.




\section{Combinatorial Preliminaries}\label{sec:prel}

\subsection*{Notation.} Our combinatorial and analytic notation is standard and follows~\cite{flajolet2009analytic}. In particular, given a formal power series of exponential type $A(x) = \sum_{n\geq0} a_n \frac{x^n}{n!}$ (EGF for short), we use the notation $[x^n] A(x)$ to indicate the $n$-th coefficient of $A(x)$. Given a bivariate function $A(x,y)$, we denote the partial derivative of $A(x,y)$ with respect to $x$ and $y$ by $A_x(x,y)$ and $A_y(x,y)$, respectively. However, we will usually use the notation $A'(x,y)$ to denote $A_x(x,y)$. We write $a_n \sim b_n$ whenever $\lim_{n \to \infty}a_n/b_n =1$. Throughout the paper $\log$ denotes the natural logarithm.
In our setting, we use the variable $x$ to mark vertices and the variable $y$ to mark edges. These variables are exponential and ordinary, respectively.

\subsection*{Graph decompositions}\label{subs:decomps}

The main ingredients in our proofs, from an enumerative combinatorial point of view, are the Symbolic Method, generating functions, connectivity decompositions, and an extension of the Dissymmetry Theorem to tree-decomposable classes.
In this section we review the essential definitions and results related to these topics.
For further details, in particular for an introduction to the Symbolic Method, we refer to the book \emph{Analytic Combinatorics} by Flajolet and Sedgewick~\cite{flajolet2009analytic}.

Let $\cC$ be a class of connected graphs with the property that a graph is in $\cC$ if and only if all its 2-connected and 3-connected components are in $\cC$.
Observe that for instance the class of connected SP graphs carrying a distinguished spanning tree shares this property.
Let $c_{n,m}$ denote the number of graphs in $\cC$ with $n$ vertices and $m$ edges.
The associated (mixed) exponential generating function (or EGF for short) is the formal power series
$$C(x,y) = \sum\limits_{m,n\geq 0} c_{n,m} \frac{x^n}{n!}y^m,$$
where $x$ and $y$ mark vertices and edges, respectively.

Similarly, let $b_{n,m}$ denote the number of 2-connected graphs in $\cC$ with $n$ vertices and $m$ edges and let $B(x,y)$ be its associated EGF.
A connected graph rooted at a vertex can be obtained from a set of rooted 2-connected graphs, where the root is not labelled and where every other vertex is substituted by a connected graph rooted at a vertex. Using the Symbolic Method, this translates into the following relation between $C(x,y)$ and $B(x,y)$ (see also~\cite{Gimenez2009}).
\begin{equation}
\label{eq:CB}
xC'(x,y)= x \exp\left(B'(xC'(x,y),y)\right).
\end{equation}

Following~\cite{Walsh1982}, a \emph{network} is defined as a simple graph with two distinguished vertices, that are called 0-pole and $\infty$-pole and do not bear a label, such that adding an edge between the two poles creates a 2-connected multigraph. If there is an edge joining the two poles, it is called \emph{root edge}.
The root edge defines the two poles, which are usually denoted by $0$ and $\infty$  (initial and final vertex of the root edge, respectively).
Let $D(x,y)$ denote the EGF associated with networks. The following equation, shown by Walsh in~\cite{Walsh1982}, reflects the relation between $B(x,y)$ and $D(x,y)$.
\begin{equation}
\label{eq:BD}
2(1+y)B_y(x,y) = x^2(D(x,y)+1).
\end{equation}
The left-hand side in Equation~\eqref{eq:BD} corresponds to the family 2-connected graph rooted at a directed edge that might not be present in the graph and the right-hand side corresponds to the family of networks (possibly empty) where in addition a label is given to the two poles.

A \emph{trivial} network consists of the two poles and of the root edge.
Following the ideas of~\cite{trakhtenbrot1958}, we further distinguish between three types of networks, namely \emph{series}, \emph{parallel}, and \emph{h-networks} as follows.
A series network $S$ can be obtained from a directed cycle with a distinguished edge (which defines the two poles of the network) by replacing every other edge by a network, and finally removing the distinguished edge.
A parallel network $P$ arises from merging at least two non-trivial networks, the root edge of each of them being not present, at their common poles.
In this family, the root edge joining the two poles of $P$ might not be present in $P$.
Finally, an $h$-network is obtained from a 3-connected graph $H$ rooted at an oriented edge, by replacing every edge of $H$ apart from the root edge by a network.
As in the parallel case, here the root edge might not be in the network.

Trakhtenbrot~\cite{trakhtenbrot1958} showed that a network is either trivial, series, parallel, or an $h$-network, and Walsh \cite{Walsh1982} translated this decomposition into counting formulas.
In SP graphs, the set of $h$-networks is empty, \textrm{so in the rest of the paper we deal only deal with series and parallel networks.}

Let us finally mention that our definition of network slightly differs from Trakhtenbrot's.
Indeed, in~\cite{trakhtenbrot1958} series networks could contain the root edge.
In our work, series networks containing the root edge (in Trakhtenbrot's sense) are always considered to be parallel.
This convention would arise to be helpful when dealing with spanning trees.

In Section~\ref{sec:trees} we aim for a precise asymptotic estimate for the number of spanning trees in random SP graphs.
For this purpose, we will enumerate the class of connected SP graphs with a distinguished spanning tree.
The main idea is to give a complete analytic analysis of the generating function associated with this class using the relations to the class of 2-connected SP graphs and to the class of networks, both carrying a distinguished spanning tree. Using Equation~\eqref{eq:BD} would imply integration steps that are known to get difficult when considering enriched classes of graphs.
Chapuy, Fusy, Kang, and Shoilekova~\cite{Chapuy2008} found, however, a convenient combinatorial trick to forgo this integration step by using the dissymmetry theorem for tree-decomposable classes (Theorem~\ref{thm:dissymmetry}) and by using the grammar for decomposing graphs into 3-connected components that they developed in~\cite{Chapuy2008}. Theorem~\ref{thm:dissymmetry} will also serve us well in Section~\ref{sec:2trees}.

A class $\cA$ of graphs is \emph{tree-decomposable} if for each graph $G \in \cA$ we can define a tree $\tau(G)$ associated with $G$.
Let $\cA_{\circ}$ denote the class of graphs $G$ in $\cA$ where $\tau(G)$ has a distinguished vertex.
Similarly, denote by $\cA_{\circ - \circ}$ the class of all graphs $G$ in $\cA$ where $\tau(G)$ carries a distinguished edge.
Finally, let $\cA_{\circ \to \circ}$ be the class of all graphs $G$ in $\cA$ where an edge of $\tau(G)$ is directed.
The Dissymmetry Theorem for trees by Bergeron~\cite{Bergeron1998} allows to express the class of unrooted trees in terms of classes of trees with a distinguished vertex, edge or with a directed edge.
This theorem can be extended to tree-decomposable classes in the following way (see e.g.~\cite{Chapuy2008}).

\begin{theorem}[Dissymmetry Theorem for tree-decomposable classes]
\label{thm:dissymmetry}
Let $\cA$ be a tree-decom\-po\-sable class of graphs. Then,
$$\cA + \cA_{\circ \to \circ} \simeq \cA_{\circ} + \cA_{\circ - \circ}.$$
\end{theorem}

Finally, let us briefly summarize Tutte's decomposition~\cite{tutte1966} for decomposing 2-connected graphs into 3-connected components.
For a thorough exposition we refer to~\cite{Chapuy2008}.

Tutte's decomposition is based on split operations and the structure obtained from this process is shown to be independent of the order of the operations.
Roughly speaking, in every split operation we split the edge set of a graph $G$ into two edge sets $E_1$ and $E_2$ that only coincide in exactly two vertices, say $u$ and $v$, and where $G[E_1]$ is 2-connected and $G[E_2]$ is connected modulo $\{u,v\}$ (meaning that there exists no partition of $E_2$ into two nonempty sets $E_2'$ and $E_2''$ such that $G[E_2']$ and $G[E_2'']$ only intersect in $u$ and $v$).
Next we add a so-called virtual edge $e$ between these two vertices.
Then we split the graph along this virtual edge which yields two graphs $G_1$ and $G_2$ that correspond respectively to $G[E_1]$ and $G[E_2]$ with $e$ now being a real edge.
We say that $G_1$ and $G_2$ are \emph{matched} by the virtual edge $e$.

The resulted structure is a collection of graphs that we call bricks.
Tutte showed that there are only three types of bricks, namely \emph{ring graphs} (R-bricks), \emph{multi-edge graphs} (M-bricks), and 3-connected graphs with at least 4 vertices (T-bricks).
The class of ring graphs is defined as the class of cyclic chains of at least 3 edges
and the class of multi-edge graphs as the class of graphs with exactly two labelled vertices that are connected by at least 3 edges. 

The \emph{RMT-tree} of a graph $G$ is defined as the graph $\tau(G)$ the vertices of which are the bricks that result from Tutte's decomposition applied to $G$. Two vertices in $\tau(G)$ are connected, when the corresponding bricks are matched by a virtual edge.
It was shown by Tutte~\cite{tutte1966} that $\tau(G)$ is indeed a tree and there are no two adjacent $R$-bricks nor two adjacent $M$-bricks.

Let $\cB$ be the class of all 2-connected graphs with at least 3 vertices.
We denote by $\cB_{R}$, $\cB_{M}$, and $\cB_{T}$ the classes of graphs $G$ in $\cB$ such that the RMT-tree associated with $G$ carries a distinguished R-vertex, M-vertex, and T-vertex, respectively.
Moreover, let $\cB_{R-M}$ denote the class of graphs $G$ in $\cB$ such that an edge between an R-vertex and an M-vertex in the RMT-tree associated with $G$ is distinguished.
The classes $\cB_{R-T}$, $\cB_{M-T}$, and $\cB_{T-T}$ are defined analogously.
Finally, let $\cB_{T\to T}$ be the class of graphs $G$ in $\cB$ such that an edge between two $T$-vertices is directed.

Using Theorem~\ref{thm:dissymmetry}, $\cB$ satisfies the following equation as shown in~\cite{Chapuy2008}:
\begin{equation}
\label{eq:dissymmetry2}
\cB \simeq \cB_R + \cB_M + \cB_T -\cB_{R-M} - \cB_{R-T} - \cB_{M-T} - \cB_{T\to T} + \cB_{T-T}.
\end{equation}

In our work we consider only SP graphs. In particular, a SP graph does not have 3-connected components. This implies that SP graphs do not contain T-bricks, and in hence RMT-trees do not have T-vertices. So Equation \eqref{eq:dissymmetry2} is simplified to
\begin{equation}
\label{eq:dissymmetry}
\cB \simeq \cB_R + \cB_M -\cB_{R-M}.
\end{equation}




\section{Analytic Background}
\label{sec:analytic}

The proofs in this paper are based on singularity analysis of generating functions.
In this section we introduce the necessary analytic background. For the sake of completeness, we state the results that we use, in particular a simplified version of the Transfer Theorems (Theorem~\ref{thm:transfer}) and a simplified version for the singularity analysis of systems of functional equations (Theorem~\ref{thm:drmota}).
For more details, we refer to the books \emph{Analytic Combinatorics} by Flajolet and Sedgewick~\cite{flajolet2009analytic} and \emph{Random Trees} by Drmota~\cite{DrmotaBook}.

Given a univariate exponential generating function
$$A(x) = \sum\limits_{n \geq 0} a_{n} \frac{x^n}{n!}$$
we would like to determine an asymptotic estimate of the sequence $(a_n)_{n\geq 0}$.
Pringsheim's Theorem (see e.g.~\cite{flajolet2009analytic}) assures that generating functions with radius of convergence $\varrho$ and non-negative Taylor coefficients have a singularity at $\varrho$, in particular a positive real dominant singularity.
As shown in~\cite{flajolet2009analytic} (Theorem IV.7), the exponential growth of the sequence $(a_n)_{n\geq 0}$ is therefore dictated by the smallest positive singularity $\varrho$ of $A(x)$ in the sense that $$[x^n]A(x) \sim \Theta(n) \varrho^{-n},$$
where $\Theta(n)$ grows subexponentially, i.e.~$\limsup_{n\to \infty} |\Theta(n)|^{1/n} = 1$.
The subexponential term $\Theta(n)$ results from the nature of this singularity.
The so-called \emph{Transfer Theorems}, developed by Flajolet and Odlyzko~\cite{flajolet1990singularity}, provide us a convenient way to determine the subexponential term of $[x^n]A(x)$.
In particular, Theorem~\ref{thm:transfer} is a special case of the Transfer Theorems in~\cite{flajolet2009analytic}.
For this, we need the definition of dented domains. Given $R, \zeta >0$ with $R > \zeta$ and $0< \phi < \pi/2$, the \emph{domain dented} at $\zeta$ (which we write as $\Delta_{\zeta}(\phi,R)$) is defined as
$$\Delta_{\zeta}(\phi,R)= \{z\in \mathbb C: |z|<R,\, z \neq \zeta,\, |\Arg(z-\zeta)|>\phi\}.$$

\begin{theorem}
\label{thm:transfer}
Let $\alpha \in \mathbb R \setminus \mathbb Z^{-}$ and let $A(x)$ be analytic in a domain $\Delta_{\varrho}(\phi,R)$ dented at the smallest positive  singularity $\varrho$ of $A(x)$.
If, as $x \to \varrho$ in $\Delta_{\varrho}(\phi,R)$,
$$A(x) \sim c\left(1-\frac{x}{\varrho}\right)^{-\alpha},$$
then
$$[x^n]A(x) = \frac{c}{\Gamma(\alpha)} n^{\alpha-1} \varrho^{-n} (1+o(1)),$$
where $\Gamma(x)$ is the classical Euler Gamma function defined as $\Gamma(x) = \int_{0}^{\infty} t^{x-1}e^{-t} dt$.
\end{theorem}

In this paper, the singular expansion of a generating function $A(x)$ in a domain dented at a singularity $\varrho$ is always of the form
$$A(x) = A_0 + A_1 X +  A_2X^2 + \ldots + A_{2k+1}X^{2k+1} + \bigO(X^{2k+2}),$$
where $X= \sqrt{1-x/\varrho}$. The even powers of $X$, being analytic functions, do not contribute to the asymptotic of $[x^n] f(x)$.

If $A_1=A_3=\dots=A_{2k-1}=0$ and $A_{2k+1}\neq 0$, then the number $(2k+1)/2$ is called the \emph{singular exponent}.
Then, by Theorem~\ref{thm:transfer} we get that
$$[x^n]A(x) \sim \frac{c}{\Gamma(\alpha)} n^{\alpha-1} \varrho^{-n}$$
with $c = A_{2k+1}$ and $\alpha = -(2k+1)/2$. When $A_1\neq 0$ we say that $A(x)$ has a square-root expansion.


Let us now turn to the asymptotic analysis of systems of functional equations.
The main reference for this topic is the paper~\cite{Drmota1997}.
We include here a shortened version (see Section 2.2.5.~in \cite{DrmotaBook} for the more general statement).
Assume that $y_1(x),\dots, y_k(x)$ are generating functions satisfying a functional system of equations.
We define $\textbf{y}=(y_1(x),\dots, y_k(x))$, and the system satisfied by $\textbf{y}$ is denoted by $\textbf{y}=\textbf{F}(x; \textbf{y})$, where $\textbf{F}=(F_1,\dots, F_k)$ is a vector of functions.
The \emph{dependency graph} $G=(V,E)$ associated with the system $\textbf{y}=\textbf{F}(x; \textbf{y})$ is an oriented graph the vertex set of which is $V=\{y_1,\dots, y_k\}$ and $\overrightarrow{y_{i}y_{j}}$ is in $E$ if and only if $\frac{\partial F_i}{\partial y_j}\neq 0$. The latter condition indicates that there is a real dependence between $F_i$ and $y_j$.
A dependency graph is said to be \emph{strongly connected} if every pair of vertices can be linked by a directed path.
Using this terminology, we can finally state the following result:

\begin{theorem}[Systems of functional equations~\cite{DrmotaBook}, simplified version]\label{thm:drmota}
Consider the functional system of equations $\textbf{y}=\textbf{F}(x; \textbf{y})$ satisfying that each $y_i$ is analytic at $x=0$. Additionally, we require that each component of $\textbf{F}$ is an entire function with positive Taylor coefficients, that it is not linear in the components $y_i$ and depends on $x$. Finally, we assume that $\textbf F(0;\textbf y)= 0$ and $\textbf F(x;\emph{\textbf{0}}) \neq 0$.
Assume also that the associated dependency graph is strongly connected. Denote by $\textbf{I}_{k}$ the $k \times k$ identity matrix and by $\mathrm{Jac}(\textbf{F})$ the Jacobian matrix associated with $\textbf{F}$ and with respect to variables $y_1,\dots, y_k$. Assume that the system of equations
\begin{equation*}
\textbf{y}=\textbf{F}(x; \textbf{y}),\,\,\quad 0= \det\left(\textbf{I}_{k}-\mathrm{Jac}(\textbf{F})\right)
\end{equation*}
has a unique solution $(x_0,\textbf{y}_0)$ in the region of analyticity of each component of $\textbf{F}$.
Then there is a unique solution $\textbf{y}$ of the initial system of equations such that the components of $\textbf{y}$ have non-negative Taylor coefficients and a square-root expansion in a domain dented at $x=x_0$.
\end{theorem}

In order to obtain asymptotic estimates we need to assure that the dominant singularity is unique in a dented domain.
This condition is usually satisfied whenever the counting formula $A(x)$ under consideration cannot be written in the form $A(x)=x^k f(x^r)$ for non-negative values $k\geq 0$ and $r\geq 2$.
More precisely, we say that a generating function $A(x)$ is \emph{aperiodic} if there exists a non-negative integer $n_0$ such that $[x^n]A(x)>0$ for $n\geq n_0$. Observe that checking the aperiodicity condition is straightforward whenever we know that for each number of vertices there exist graphs in the family under study.
The generating functions we consider in the forthcoming section (which are defined by an implicit equation, or by means of Theorem~\ref{thm:drmota}) will satisfy the aperiodicity condition by obvious combinatorial reasons. This will imply uniqueness of the dominant singularity.
See \cite{DrmotaBook} for details.




\section{Spanning trees in series-parallel graphs}
\label{sec:trees}

In this section we present the proofs of Theorem~\ref{thm:main} determining a precise asymptotic estimate for the expected number of spanning trees in random connected SP graphs. Additionally, we elaborate the expected number of spanning trees in a random 2-connected SP graph of a given edge density.
Recall that all graphs in this paper are considered to be labelled.
In order to count spanning trees in SP graphs, we are concerned with the enumeration of SP graphs carrying a distinguished spanning tree.
For this, let $c_{n,m}$ and $b_{n,m}$ now denote the number of connected, respectively 2-connected, SP graphs with a distinguished spanning tree and let $C(x,y)$ and $B(x,y)$ be their associated counting formula, where $x$ and $y$ mark again vertices and edges, respectively.
Furthermore, let $\cD$ denote the class of series-parallel networks carrying a distinguished spanning tree and let $D(x,y)$ be its associated generating function.

\subsection{Expected number of spanning trees in random SP graphs}
\label{sec:1stmoment}

The first step in our proof of Theorem~\ref{thm:main} is the enumeration of SP networks that carry a distinguished spanning tree.
For the purpose of counting spanning trees in SP networks, we need to introduce the following auxiliary class.
Let $\cbD$ denote the class of SP networks that carry a distinguished spanning forest with two components, each of which contains one of the poles. Let $\bD(x,y)$ be its associated EGF.

Recall that a network is either trivial, series or parallel.
By convention, we assume that networks with a root edge are parallel.
Therefore, we define the following classes of networks. Let $\cS$ and $\cbS$ denote the class of series networks that carry a distinguished spanning tree, respectively a distinguished spanning forest with two components each of which contains one of the poles.
We denote their associated EGFs by $S(x,y)$ and $\bS(x,y)$, respectively.
Similarly, let $\cP$ and $\cbP$ denote the class of parallel networks and which carry a distinguished spanning tree, respectively a distinguished spanning forest with two components, each of which contains one of the poles.
Observe that in both families the root edge might be present.
Their associated EGFs are denoted by $P(x,y)$ and $\bP(x,y)$, respectively.
For the sake of readability we may omit the parameters whenever they are clear from the context.

We start with elaborating relations between $D$, $\bD$, $S$, $\bS$, $P$, and $\bP$ in order to obtain a suitable system of equations. One can easily verify that
\begin{equation}
\label{eq:D1}
D(x,y) = y + S(x,y) + P(x,y),
\end{equation}
and
\begin{equation}
\label{eq:D2}
\bD(x,y) = y + \bS(x,y) + \bP(x,y).
\end{equation}
Note that in Equation~\eqref{eq:D1} the variable $y$ on the right-hand side corresponds to a trivial network with a distinguished spanning tree, whereas in Equation~\eqref{eq:D2} it corresponds to a trivial network with a distinguished spanning forest that consists of two components of size 1.

Let us now analyse series networks. Observe that a series network $N$ can be decomposed into at least two networks, where the 0-pole of the $i$-th network is identified with the $\infty$-pole of the $(i+1)$-th network. Equivalently, $N$ can be decomposed into an ordered sequence formed by a network $N'$ that is not series and an arbitrary network $N''$ that are joined by a series operation. If $N \in \cS$, then each of these two networks contains a distinguished spanning tree. Therefore, we have
\begin{equation}
\label{eq:S1}
S(x,y)= \big(D(x,y)-S(x,y)\big)x D(x,y)= \big(y+P(x,y)\big)xD(x,y).
\end{equation}

\begin{figure}[htb]
\begin{center}
\includegraphics[width=6 cm, page=1]{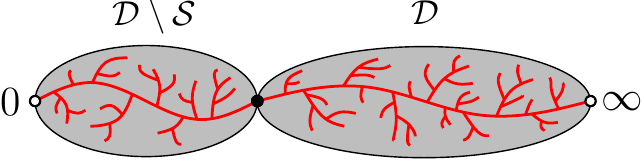}
\caption{Decomposition of $N\in \cS$.}
\label{fig:series}
\end{center}
\end{figure}

If $N \in \cbS$, then either $N' \in \cD \backslash \cS$ and $N'' \in \cbD$, or $N' \in \cbD \backslash \cbS$ and $N'' \in \cD$. This translates into the following equation:
\begin{eqnarray}\label{eq:S2}
\bS(x,y) &=& \big(D(x,y)-S(x,y)\big)x \bD(x,y)+  \big(\bD(x,y)-\bS(x,y)\big)xD(x,y) \nonumber \\
         &=&  \big(y+P(x,y)\big)x\bD(x,y) + \big(y+\bP(x,y)\big)xD(x,y).
\end{eqnarray}

\begin{figure}[htb]
\begin{center}
\includegraphics[width=6 cm, page=2]{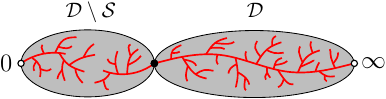}
\quad \quad
\includegraphics[width=6 cm, page =3]{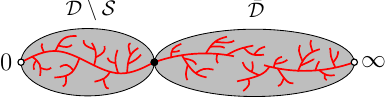}
\caption{Decomposition of $N \in \cbS$.}
\label{fig:series}
\end{center}
\end{figure}

Finally, we analyse parallel networks. A parallel network can be described as a set of at least one series network if the root edge is present, or of at least two series networks, otherwise. If $N \in \cP$, we need to distinguish between the case that the root edge is present and the case that it is not. In the second case, all series networks are in $\cbS$ except for one which is in $\cS$. If in the first case the root edge is in the distinguished spanning tree of $N$, then all series networks are in $\cbS$. If, on the other hand, the root edge is not in the spanning tree, then exactly one of the series networks is in $\cS$ and all other networks are in $\cbS$.
Thus, we get
\begin{equation}
\label{eq:P1}
P(x,y)=y\big(\exp(\bS(x,y))-1\big)+ y\big(S(x,y) \exp(\bS(x,y))\big) + S\big(\exp(\bS(x,y))-1\big).
\end{equation}
\begin{figure}[htb]
\begin{center}
\includegraphics[width=3.7 cm, height=3.7cm, page=3]{parallel.pdf}
\quad \quad
\includegraphics[width=3.7 cm, height=3.7cm,page =1]{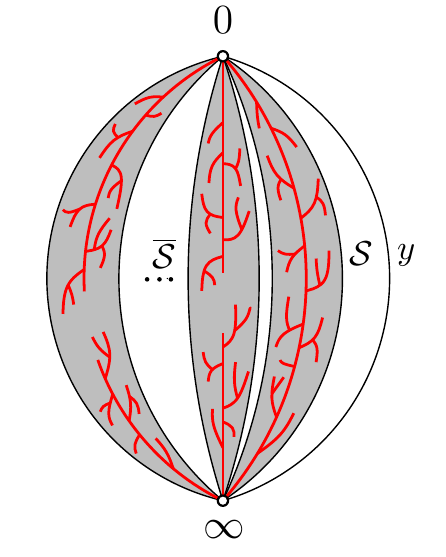}
\quad \quad
\includegraphics[width=3.7 cm, height=3.7cm, page =2]{parallel.pdf}
\caption{Decomposition of $N \in \cP$.}
\label{fig:series}
\end{center}
\end{figure}
If $N \in \cbP$, then $N$ can be decomposed into the root edge if present and into series networks in $\cbS$ that are joined by a parallel operation. If the root edge is present, then there is at least one other network. In the other case, there must be at least two. This gives rise to the following equation:
\begin{equation}
\label{eq:P2}
\bP(x,y) = \big(\exp(\bS(x,y))-\bS(x,y)-1\big) + y\big(\exp(\bS(x,y))-1\big).
\end{equation}

\begin{figure}[htb]
\begin{center}
\includegraphics[width=3.8 cm, height=3.7cm, page=8]{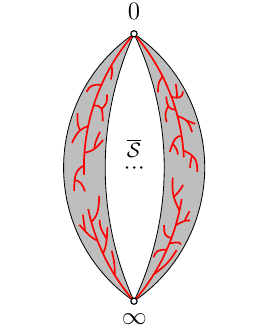}
\quad \quad
\includegraphics[width=3.8 cm, height=3.7cm,page =9]{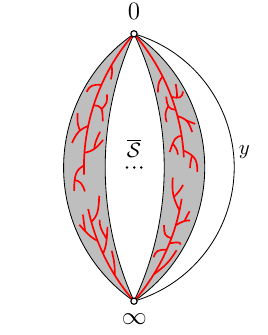}
\caption{Decomposition of $N \in \cbP$.}
\label{fig:series}
\end{center}
\end{figure}

Using formal manipulations, we get that the system of Equations~\eqref{eq:D1}-\eqref{eq:P2} defines the following implicit expression for $D(x,y)$:
\begin{equation}
\label{eq:DD1}
D= \left(y+(1+y)\frac{x D^2}{1+xD}\right) \exp\left(-xD\frac{(y(1+xD)-(1+y)D)(2+xD)}{(y(1+xD)+(1+y)xD^2)(1+xD)^2}\right).
\end{equation}
In order to study the singular behaviour of all the previous generating functions we could apply Dromta-Lalley-Woods methodology for systems of functional equations (see e.g.~\cite{flajolet2009analytic}). However, as in this particular case we have an expression for $D(x,y)$ not depending on any other variables but $x$ and $y$, we will analyse Equation~\eqref{eq:DD1} in order to get the singular behaviour of $D(x,y)$. The following theorem is reminiscent to Lemma 3.3. in~\cite{GiNoyRue2013}:

\begin{lemma}\label{lem:main_D}

Let $D(x,y)$ be the formal power series defined by the equation $\Phi(x,y;D(x,y))=0$, where
\begin{equation*}
\Phi(x,y;z)=z-\left(y+(1+y)\frac{ x z^2}{1+xz}\right) \exp\left(-xz\frac{(y(1+xz)-(1+y)z)(2+xz)}{(y(1+xz)+(1+y)xz^2)(1+xz)^2}\right).
\end{equation*}
Then, for every choice of $y>0$ it holds that $D(x,y)$ has a unique square-root singularity $R(y)$ such that $D(x,y)$ has a singular expansion of the following form in a domain dented at $x=R(y)$:
\begin{equation}\label{eq:sing_D}
D(x,y)= D_0(y) +D_1(y) X(y) + D_2(y) X(y)^2 + D_3(y) X(y)^3 + \bigO(X(y)^4),
\end{equation}
where $X(y)= \sqrt{1-x/R(y)}$. In particular, for $y=1$ we have the numerical values $x=R(1)=R\approx 0.05668$, $D_0(1) \approx 1.82404$ $D_1(1) \approx -1.52769$, $D_2(1) \approx 1.34779$ and $D_3(1) \approx-1.25138$.
\end{lemma}

\begin{proof} Fix $y>0$. A simple computation shows that $\Phi_{z}(0,y;D(0,y))=1>0$ and $D(0,y)=y$.
Hence, by the Implicit Function Theorem, $D(x,y)$ is analytic at $x=0$.

We start with showing that  $D(x,y)$ has a finite radius of convergence.
Denote the singularity of the function $D(x,y)$ by $R(y)$.
Observe that $ [x^n] D_{\emptyset}(x,y)< [x^n]D(x,y)$, where $D_{\emptyset}(x,y)$ is the generating function associated with SP networks without a distinguished spanning tree.
As it is shown in \cite{Bodirsky2007}, the radius of convergence $R_{\emptyset}(y)$ of $D_{\emptyset}(x,y)$ is finite.
In particular, $0<R(y)\leq R_{\emptyset}(y)<1<\infty$ and $D(x,y)$ ceases to be analytic at $x=R(y)$.

Observe that the only source of singularity for $D(x,y)$ is the condition $\Phi_z(R(y),y; D(R(y),y))=0$, meaning that the singularity arises from a branch point.
Let us now justify that we have $\Phi_{zz}(R(y),y; D(R(y),y))\neq 0$, which would give the claimed square-root expansion.
This condition is enough in order to assure, for each choice of $y$, and square-root type singularity.
For a contradiction, let us assume the opposite. Hence, we have a solution $(R_0,y_0,z_0)$ of the following system of equations:
\begin{equation*}
\Phi(x,y; z)=0,\,\, \,\Phi_{z}(x,y;z)=0,\,\,\,\Phi_{zz}(x,y; z)=0.
\end{equation*}
Observe that $\Phi(x,y;z)= z- A(x,y;z) \exp(B(x,y;z))$ with $A(x,y;z)$ and $B(x,y;z)$ being rational functions.
Hence, $\Phi_z(x,y;z)=1- C(x,y;z)\exp(B(x,y;z))$ where again $C(x,y;z)$ is a rational function.
Finally, $\Phi_{zz}(x,y;z)$ can be written in the form $E(x,y;z)\exp(B(x,y;z))$ for a certain rational function $E(x,y;z)$.

In particular, combining the first two equations by eliminating the exponential term, we get the following system of rational equations:
$$zC(x,y;z)=A(x,y;z),\, \, \, E(x,y;z)=0.$$
After rearranging the denominators in both expressions, such a system can be transformed into a system of two polynomial equations $P_1(x,y;z)=0,\,P_2(x,y;z)=0$, from which we can get a new polynomial equation $Q(x,y)=0$ by eliminating the variable $z$. By carrying out the explained computations with \texttt{Maple}, we obtain
$$
Q(x,y) = (-4y+yx-4)y(y+1) T(x,y),
$$
where
\begin{eqnarray*}
T(x,y)= 100(1+y)^4+ 6917y(1+y)^3x+1266y^2(1+y)^2x^2-1867y^3(1+y)x^3+280y^4x^4.
\end{eqnarray*}
We now argue that $Q(x,y)=0$ does not have a solution with both $y>0$ and $x<1$. Observe that the first multiplicative term $-4y+yx-4$ gives the solution $(x,y)=(4+4/y,y)$. This means in particular that $x$ is always greater than $1$. It is also obvious that the multiplicative terms $y$ and $y+1$ cannot contribute with the required solution. Therefore, we need to analyse the existence of solutions $(x,y)$ of $T(x,y)$ with the condition $y>0$ and $x<1$. Using that $y(1+y)^3>y^3(1+y)$, $x>x^3$ for all $y>0$ and $0<x<1$, we know that $6917y(1+y)^3x>6917 y^3(1+y)x^3>1867y^3(1+y)x^3$. Hence, $T(x,y)=0$ does not have solutions with both $0<x<1$ and $y>0$, which implies that the solution $(x_0,y_0,z_0)$ of the equation $\Phi(x,y;z)=\Phi_z(x,y;z)=0$ satisfies that $\Phi_{zz}(x_0,y_0;z_0)\neq 0$.
Hence, the singularity of $D(x,y)$ is of a square-root type in a domain dented at $x=R(y)$. This proves the singular expansion in Equation~\eqref{eq:sing_D}.

In order to prove the special case of $y=1$ in the statement of the lemma, we set $y=1$, $R=R(1)$, and  $X=X(1)$ (and consequently $x=R(1-X^2)$). By plugging the singular expansion of $D(x,1)$ 
 in $\Phi(x,y;z)=0$, taking the Taylor expansion in terms of $X$, and applying the method of indeterminate coefficients, we get the numerical values as claimed.
Finally, observe that for each choice of $y>0$, the generating function $D(x,y)$ is aperiodic, as for every $n$ there exists a network on $n$ vertices.
This gives that the singularity $R(y)$ is unique, and the result holds.
\end{proof}

Knowing that $D(x,y)$ admits a singular expansion of square root-type in a domain dented at $x=R(y)$, one can compute by means of indeterminate coefficients the exact expressions of $D_i(y)$ for $i\geq 1$ in terms of the function $D(R(y),y)=D_0(y)$, which satisfies the functional equation $\Phi(R(y),y,D_0(y))=0$.
Although the expressions are long, we need to compute for enumerative purposes the evaluations at $y=1$. For computational purposes, we include the following lemma where the coefficients of the singular expansions (rounded up to 5 digits) of the EGF $\bD(x,1)$, $S(x,1)$, $\bS(x,1)$, $P(x,1)$, and $\bP(x,1)$ are obtained.
Despite that in order to get asymptotic estimates for these counting formulas we only need the multiplicative constant of the term $(1-x/R)^{1/2}$, in order to get the precise asymptotic in the 2-connected level we need expansions up to term $(1-x/R)^{3/2}$.

\begin{lemma}
\label{lem:singD} For each value of $y>0$ the generating functions $\bD$, $S$, $\bS$, $P$, and $\bP$ have a square-root singular expansion in a domain dented at $R(y)$, where $R(y)$ is the unique singularity of $D(x,y)$. Furthermore, for $y=1$ the singular expansions (with rounded coefficients) of $\bD$, $S$, $\bS$, $P$, and $\bP$ in a domain dented at $R \approx 0.05668 $ are
\begin{align*}
\bD(x,1)&= \bD_0(1)+ \bD_0(1) X+ \bD_2(1) X^2+ \bD_3(1) X^3    + \bigO(X^4)\\
S(x,1)&= S_0(1) +S_1(1) X + S_2(1) X^2+ S_3(1) X^3+  \bigO(X^4)\\
\bS(x,1)&= \bS_0(1) +\bS_1(1) X + \bS_2(1) X^2 +\bS_3(1) X^3+ \bigO(X^4)\\
P(x,1)&= P_0(1) +P_1(1) X  +P_2(1) X^2 +P_3(1) X^3+   \bigO(X^4)\\
\bP(x,1)&= \bP_0(1) +\bP_1(1) X + \bP_2(1) X^2 +\bP_3(1) X^3+ \bigO(X^4),
\end{align*}
where  $X= \sqrt{1-x/R}$, and the constants have the following approximate values:

\begin{center}
  \begin{tabular}{c|cccc}

              & $i=0$  & $i=1$ & $i=2$ & $i=3$ \\
    \hline
    $\bD_i(1)$ & $1.71871$ & $-1.17120$ & $1.17120$ & $-0.59820$ \\
     $S_i(1)$  &  $0.17092$     & $-0.27289$        & $0.18433$      & $- 0.15440$  \\
        $\bS_i(1)$  &  $0.30701$     & $-0.43079$        & $0.19616$      & $-0.12220$  \\
        $P_i(1)$  &  $0.65312$     & $-1.25480$        & $1.16347$      & $-1.09697$  \\
        $\bP_i(1)$  &  $0.41170$     & $-0.74041$        & $0.58941$      & $-0.47600$  \\
  \end{tabular}
  \end{center}
\end{lemma}
\begin{proof} The first claim follows directly due to Equations \eqref{eq:D1}-\eqref{eq:P2}, which are analytic and allow us to express $\bD$, $S$, $\bS$, $P$, and $\bP$ in terms of $D$. In particular, all these generating functions have a unique singularity at $x=R(y)$. The second part follows by setting $y=1$ and by plugging the singular expansion of $D(x,1)$ into Equations \eqref{eq:D1}-\eqref{eq:P2}.
\end{proof}

Now we turn to the analysis of $B(x,y)$, the EGF associated with the class of 2-connected SP graphs carrying a distinguished spanning tree. In our context, Equation~\eqref{eq:BD} translates to
\begin{equation}\label{eq:BD2}
2(1+y)B_y(x,y)=x^2\Big(1+D(x,y)+\bD(x,y)-y\big(\exp(\bS(x,y)-1)\big)\Big)
\end{equation}
which roughly speaking means that when marking an edge in a 2-connected SP graph with a distinguished spanning tree, the resulting object is a network either of type $\cD$ or $\cbD$, but not a parallel network of type $\cP$ with an edge linking the poles (see Equation \eqref{eq:P2}).
A direct integration of Equation~\eqref{eq:BD2} is technically involved due to the relations between the generating functions associated with the different types of networks.
However, we can get a simple expression of $B(x,y)$ in terms of the EGF associated with the networks just by combinatorial arguments. In the following lemma we provide such an equation.
\begin{lemma}
\label{lem:dissymmetryB}
The EGF $B(x,y)$ associated with the class of 2-connected SP graphs carrying a distinguished spanning tree can be expressed as
\begin{equation}
\label{eq:dissymmetryB}
B(x,y) = \frac{x^2}{2}y+ B_R(x,y) + B_M(x,y) - B_{R-M}(x,y),
\end{equation}
where
\begin{align}
B_R(x,y) &= \frac{x^2}{2}S(\bD-\bS), \label{eq:BR}\\
B_M(x,y) &= \frac{x^2}{2}\Big(S \big(\exp(\bS)-\bS-1\big)+ yS \big(\exp(\bS)-1\big) + y \big(\exp(\bS)-\bS-1\big)\Big), \label{eq:BM}\\
B_{R-M}(x,y) &= \frac{x^2}{2}(S\bP+\bS P). \label{eq:BRM}
\end{align}
\end{lemma}
\begin{proof}
Applying Tutte's decomposition to 2-connected SP graphs bearing a distinguished spanning tree on at least 3 vertices only yields R-bricks (ring graphs) and M-bricks (multi-edge graphs), both carrying a distinguished spanning tree. In particular, there are no T-bricks since the set of $h$-networks is empty in our case.
We obtain Expression~\eqref{eq:dissymmetryB} for $B(x,y)$ using Equation~\eqref{eq:dissymmetry} to which we needed to add $x^2y/2$ since we also consider a single edge to be a 2-connected SP graph.

Let us study each term.
Let $R$ be a distinguished R-brick with a distinguished spanning tree.
By definition, $R$ is a cyclic chain of at least 3 networks that carries a spanning tree. In particular, exactly one of these networks is in $\cbD$ while the other ones are in $\cD$. This means that $R$ can be decomposed into a non-series network in $\cbD$ and a series network in $\cS$ that are joined by a parallel operation and where the two poles are added to the graph, see also Figure~\ref{fig:Rbrick}.
This gives Equation~\eqref{eq:BR}.

\begin{figure}[htb]
\begin{center}
\includegraphics[width=5 cm, page=1]{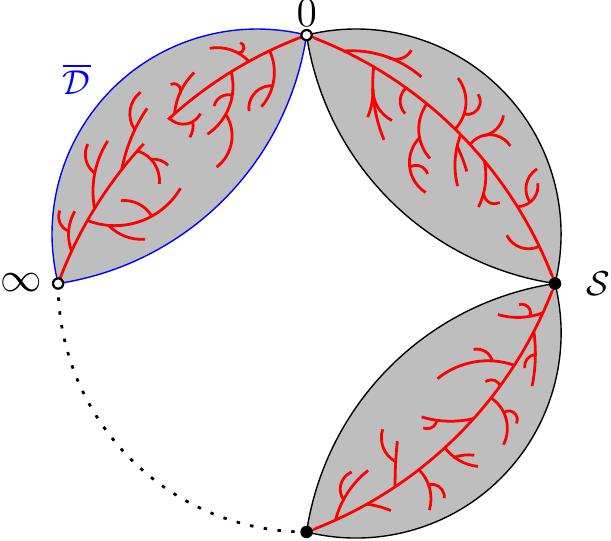}
\caption{Decomposition of a distinguished R-brick in the RMT-tree.}
\label{fig:Rbrick}
\end{center}
\end{figure}

We continue with M-bricks. Let $M$ be a distinguished M-brick with a distinguished spanning tree.
Then, $M$ can be decomposed into at least three networks, all but possibly one of which are series and the possibly other one a single edge.
These networks are joined by a parallel operation and the two poles are again added to the graph. This situation is similar to the decomposition of parallel networks carrying a spanning tree that we considered for developing Equation~\eqref{eq:P1}. The main difference is that, by definition, $M$ is decomposed into at least three and not into at least two networks. We need to distinguish again between the two cases where there is a single edge component in $M$ and where there is no such a component.
Observe that it is not possible that there are two such components in $M$ since we are only considering simple graphs.
In the former of the two cases, we note that if the edge of the single edge component is contained in the distinguished spanning tree, then exactly one of the series networks is in $\cS$ while all the others are in $\cbS$. If, on the other hand, the edge is not in the spanning tree, then all series networks must be in $\cbS$. This gives rise to Equation~\eqref{eq:BM}.

Finally, we need to decompose 2-connected SP graphs with a distinguished spanning tree and with a distinguished $\{R,M\}$-edge in the RMT-tree. This means, that the distinguished edge corresponds to a virtual edge $\{x,y\}$ matching a R-brick and a M-brick. Hence the graphs can be decomposed into a series network and a parallel network by a parallel operation, where we need to add again the two poles to the graph. One of the two networks must be in $\cD$ while the other one must be in $\cbD$. Hence, Equation~\eqref{eq:BRM} holds. See Figure~\ref{fig:RM} for an illustration of this situation.
\end{proof}

\begin{figure}[htb]
\begin{center}
\includegraphics[width=4.5 cm, page=2]{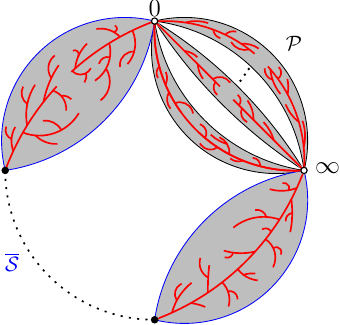}
\quad \quad
\includegraphics[width=4.5 cm, page=3]{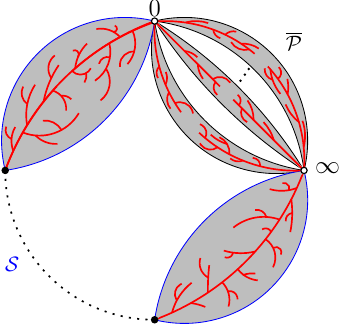}
\caption{Decomposition of a distinguished $\{R,M\}$-edge in the RMT-tree.}
\label{fig:RM}
\end{center}
\end{figure}

We can now analyse the singular behaviour of $B(x,y)$.

\begin{lemma}
\label{lem:singB}
Let $y>0$. Then, $B(x,y)$ has a unique square-root singularity, which is the unique singularity $R(y)$ of the function $D(x,y)$ in Lemma~\ref{lem:main_D}. Moreover, $B(x,y)$ has a singular expansion of the following form in a domain dented at $x=R(y)$:
\begin{equation}\label{eq:sing_B}
B(x,y)= B_0(y)  +B_2(y) X(y)^2 +B_3(y) X(y)^3 + \bigO(X(y)^4),
\end{equation}
where $X(y)= \sqrt{1-x/R(y)}$. When $y=1$ we have the numerical values $x=R(1)=R\approx 0.05668$, $B_0(1) \approx 0.00176$, $B_2(1) \approx -0.00394$ and $B_3(1) \approx 0.00062$.
\end{lemma}

\begin{proof}
Observe that the generating functions $B_R(x,y)$, $B_M(x,y)$ and $B_{R-M}(x,y)$ are analytic transformations of the generating functions for networks (namely, the EGFs in Lemma~\ref{lem:dissymmetryB}).
Hence, $B(x,y)$ has a unique dominant singularity which is the same one as the coinciding singularity of EGFs of Lemma~\ref{lem:singD}, namely $R(y)$.
Similarly, for each $y$, $B(x,y)$ admits a singular expansion in a domain dented at $R(y)$.
In order to obtain it, we express the singular expansion of each of the network exponential generating functions appearing in Equation~\eqref{eq:dissymmetryB} in terms of the singular expansions obtained in Lemma~\ref{lem:singD}. Observe that Equation~\eqref{eq:BD2} implies that the singular expansion of $B(x,y)$ must start at $X(y)^3$, which gives in particular that $B_1(y)=0$ (see Theorem VI.9 from \cite{flajolet2009analytic}).

Finally, by setting $y=1$ and by the same procedure as above using \texttt{Maple}, we obtain the approximation of $B_i(1)$ for $i\geq 0$ as stated in the lemma. In particular, the term $B_3(1)$ depends on all singular coefficients in Lemma~\ref{lem:singD}.
\end{proof}
Finally, we analyse the generating function $C(x,y)$ of connected SP graphs carrying a distinguished spanning tree. Since the singular expansion of $B(x,y)$ is of a square-root type with exponent $1/2$ as it is shown in Equation~\eqref{eq:sing_B}, we get the singular expansion of $C(x,y)$  immediately from Proposition~3.10 in \cite{GiNoyRue2013} (see also \cite{Drmota2013}).
\begin{lemma}
\label{lem:singC}
The singularity of $C(x,y)$ is at $\brho(y) = \tau(y)/ \exp(B_x(\tau(y),y)),$ where $\tau(y)$ is the unique solution of the equation $\tau(y)B_{xx}(\tau(y),y) = 1$. The singular expansion of $C(x,y)$ in a domain dented at $\brho(y)$ is
$$C(x,y) = C_0(y)+ C_2(y) X(y)^2 + C_3(y)X(y)^3 + \bigO(X(y)^4),$$
where $X(y) = \sqrt{1-x/\brho(y)}$ and
\begin{align*}
C_0(y)&= \tau(1+ \log \brho(y) - \log \tau(y)) + B(\tau(y),y) ,\\
C_2(y)&= - \tau(y), \text{ and}\\
C_3(y)&= \frac{3}{2} \sqrt{\frac{2\brho(y) \exp(B_{x}(\brho(y),y))}{\tau B_{xxx}(\tau(y),y) - \tau B_{xx}(\tau(y),y)^2 + 2 B_{xx}(\tau(y),y)}} .
\end{align*}

Additionally, when $y=1$ we have $\brho(1)\approx 0.05288$, $C_0(1) \approx 0.05450$, $C_2(1)= - \tau \approx -0.05668$,  and $C_3(1)\approx  0.00145$.
\end{lemma}

\begin{proof} See the proof of Proposition 3.10 in \cite{GiNoyRue2013} for the analysis for a general value of $y$. When $y=1$, we use \texttt{Maple} to obtain the approximations of the constants. The unicity of the singularity is assured by the aperiodicity of $C(x,y)$ with $y$ being fixed. See for instance the proof of \cite[Lemma~7]{DrFuKaKrRu11} and \cite[Lemma~9]{DrFuKaKrRu11}.
\end{proof}

Now we have all necessary ingredients to prove the main theorem of this section.

\begin{proof}[Proof of Theorem~\ref{thm:main}]
We will prove the statement for $X_n$ in detail.
The result for $Z_n$ is obtained \emph{mutatis mutandis}. Let us denote by $\mathcal{C}_n$ the set of all connected SP graphs on $n$ vertices, and $\mathcal{C}^s_n$ the set of all connected SP graphs on $n$ vertices carrying a distinguished spanning tree. For a graph $G\in \mathcal{C}_n$ we write $s(G)$ for the number of spanning trees in $G$. Then, the expected value of $X_n$ can be written as:
\begin{equation}\label{eq:expectation}
\mathbb{E}[X_n]=\sum_{G\in \mathcal{C}_n} s(G) \mathbb{P}[G]=\frac{\sum_{G\in \mathcal{C}_n}s(G)}{|\mathcal{C}_n|}=\frac{|\mathcal{C}^s_n|}{|\mathcal{C}_n|}=\frac{[x^n]C(x,1)}{|\mathcal{C}_n|}.
\end{equation}
It follows directly from Lemma~\ref{lem:singC} and Theorem~\ref{thm:transfer} that the number of connected SP graphs on $n$ vertices that carry a distinguished spanning tree is asymptotically equal to $\frac{C_3(1)}{\Gamma(-3/2)}n^{-5/2} \brho(1)^{-n} n!$. The number of connected SP graphs on $n$ vertices is asymptotically equal to $c_s n^{-5/2}\varrho_s^{-n}n!$, where $c_s \approx 0.0067912$ and $\varrho_s \approx 0.11021$ are computable constants, as shown in~\cite[Theorem 3.7]{Bodirsky2007}. Dividing the former by the latter as in Equation~\eqref{eq:expectation}, we obtain that the expected value of $X_n$ is asymptotically equal to $s \varrho^{-n}$, where $s\approx 0.09063$ and $\varrho^{-1} \approx 2.08415$.

The corresponding result for 2-connected SP graphs is obtained analogously by using Theorem~2.6 of~\cite{Bodirsky2007}, which states that the number of 2-connected SP graphs on $n$ vertices is asymptotically equal to $b n^{-5/2} r^{-n} n!$, where $b \approx 0.00101$ and $r\approx 0.12800$.
\end{proof}

\subsection{Fixing the edge density and limiting distributions} \label{sub:fix-y}

The previous results can be used to study random SP graphs with a fixed edge density, as well as limiting distributions for the number of edges.
For sake of brevity, we only discuss the family of 2-connected SP graphs, but similar observations hold in the connected case.
The first main important observation is that the number of edges in a uniformly at random 2-connected SP graph carrying a distinguished spanning tree follows a normal limiting distribution:
Lemma~\ref{lem:singB} shows that the singular behaviour of $B(x,y)$ is the same when choosing $y$ in a real-valued neighbourhood of $1$. Then, by the Quasi-Powers Theorem by Hwang (see \cite{Hwang98}) the distribution follows a normal limit law with linear expectation and variance. In particular, the number of edges is concentrated around its mean value. This behaviour is similar to the case of 2-connected SP graphs (without a spanning tree), where again the number of edges is normally distributed (see \cite{Bodirsky2007}).

Under these circumstances, our techniques also provide a method to study the expected number of spanning trees in a random SP graph on $n$ vertices of a given edge density $\mu$.
Following the arguments of \cite[Theorem 3]{Gimenez2009}, for every $\mu>0$ we can choose a value $y_0 > 0$ such that if we assign the weight $y_0^k$
to each graph with $k$ edges, then only the graphs with $n$ vertices and with approximately $\mu n$ edges (with a deviation of order $n^{1/2}$) have non-negligible weight.
Such technique is valid whenever Quasi-Powers Theorem holds, and hence we can apply it in our context.

As a case example, we plot the expected value of the random variable $X_{n,\mu}$
that counts the number of spanning trees in a graph chosen uniformly at random from the 2-connected SP graphs with $n$ vertices and edge density $\mu$. 
Let $R(y)$ denote the radius
of convergence of $B(x,y)$.
Given an edge density $\mu$, the right choice for $y_0$ is the unique positive solution of the following equation (see e.g.~\cite[Theorem 3]{Gimenez2009}):
\begin{equation}\label{eq:mu}
-y_0 \frac{R_y(y_0)}{R(y_0)}=\mu.
\end{equation}
Observe that when $\mu$ tends to $1$, the family of SP graphs under study are graphs with a small but positive number of cycles,
whereas when $\mu$ tends to $2$, the subfamily under study tends to the class of $2$-trees.
These cases correspond to the ones when $y$ tends to 0 and infinity, respectively.
Both cases will be analyzed fully in detail in Sections \ref{sec:2trees} and \ref{sec:fixedexcess}, respectively.

\begin{figure}[htb]
\begin{center}
\includegraphics[width=7.7 cm, height= 7cm, trim= 3cm 10cm 5cm 1cm]{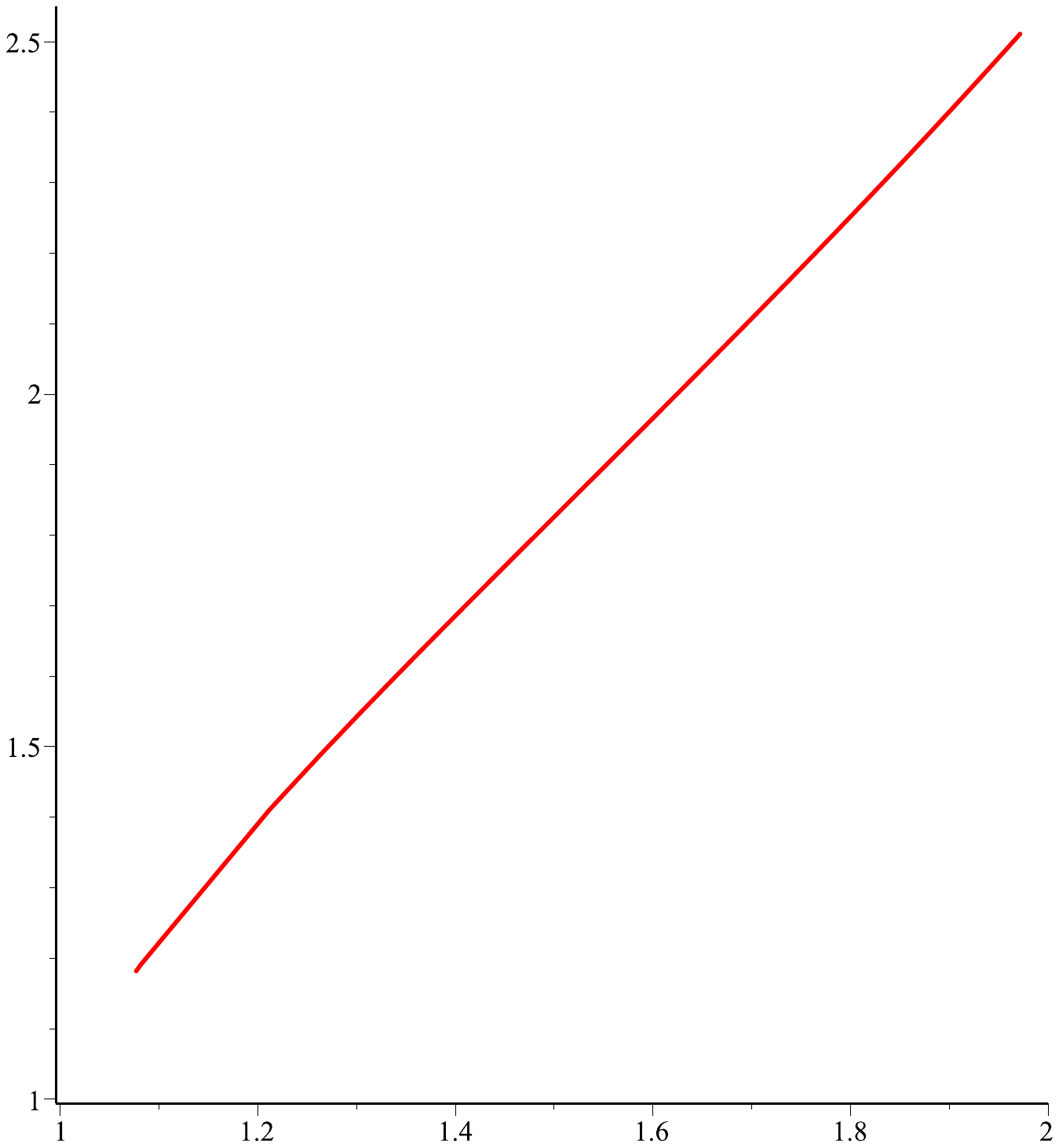}
\caption{Exponential growth constant of the expected number of spanning trees  in a random 2-connected SP graph (ordinate) as a function of its edge density (abscissa).}
\label{fig:plot}
\end{center}
\end{figure}

The precise computational method to obtain the exponential growth constant of the expected value of $X_{n,\mu}$ as a function of the edge density is the following.
For a given density $\mu$ we use~\eqref{eq:mu} to obtain the corresponding $y_0$. Then we use the implicit expression of the singularity curve stated in Theorem 2.2.~of~\cite{Bodirsky2007} in order to obtain the growth constant of the number of 2-connected SP graphs of edge density equal to $\mu$. To get the growth constant in the setting of 2-connected SP graphs carrying a spanning tree and with edge density $\mu$, we perform the calculations explained in the proof of Lemma~\ref{lem:main_D} for $y_0$.
Finally, the exponential growth of the expected value of $X_{n,\mu}$ is obtained by dividing these two numerical values as we did in the proof of Theorem~\ref{thm:main}.
In Figure~\ref{fig:plot} we plot this exponential growth constant in terms of the edge density $\mu\in (1.07626, 1.97173)$.

Let us mention that the non-plotted margins for $\mu$ correspond to values of $y$ very close to $0$ and when $y$ tends to infinity. In both cases, the numerical method used to get the constant growth for the number of spanning trees fails because of indetermination of the operation to be carried.
Detailed analysis of the two cases when the edge density reaches its maximum and when it tends to its minimum will be carried out in Section~\ref{sec:2trees} and in Section~\ref{sec:fixedexcess}, respectively.



\subsection{Variance of the number of spanning trees in random SP graphs}
\label{sec:2ndmoment}

Refining the combinatorics exploited in the proof of Theorem~\ref{thm:main}, one has also access to the second moment of the random variables $X_n$ and $Z_n$. In this subsection we develop this by determining the growth constant of the variance of $Z_n$. This will also show that $Z_n$ is not concentrated around its expected value. Recall that $Z_n$ was defined as the random variable which counts the number of spanning trees in a random 2-connected SP graph on $n$ vertices.

In order to determine the growth constant of the second moment of $Z_n$, we will first study the asymptotic behaviour of the number of 2-connected SP graphs on $n$ vertices carrying two distinguished spanning trees.
As in Subsection~\ref{sec:1stmoment}, we start with the analysis of networks carrying spanning trees and spanning forests.
We define $\cD^\ast$, $\cS^\ast$, and $\cP^\ast$ as the classes of SP, series, and parallel networks, respectively, each carrying two distinguished spanning trees.
Let $D^\ast(x,y)$, $S^\ast(x,y)$, and $P^\ast(x,y)$ denote their EGFs.
In order to be able to analyse these functions, we need again some auxiliary classes. Let $\ctD$ denote the class of all SP networks carrying a distinguished spanning tree and a distinguished spanning forest with (exactly) two components each of which contains one pole. Furthermore, let $\chD$ denote the class of all SP networks carrying two distinguished spanning forests both with (exactly) two components each of which contains one pole.
Let $\tD(x,y)$ and $\hD(x,y)$ denote their EGFs. In the same way $\ctS$, $\ctP$, $\chS$, and $\chP$ as well as $\tS(x,y)$, $\tP(x,y)$, $\hS(x,y)$, and $\hP(x,y)$ are defined. We might again omit the parameters whenever they are clear from the context. Following the proof of Theorem~\ref{thm:main}, we start with the following lemma that provides the growth constant of these generating functions.

\begin{lemma}\label{lem:varD}
For $y=1$, the generating functions $D^\ast$, $\tD$, $\hD$, $S^\ast$, $\tS$, $\hS$, $P^\ast$, $\tP$, and $\hP$ have a square-root expansion in a domain dented at $R_2\approx 0.02407$.
\end{lemma}
\begin{proof}
We start again with elaborating relations between all given generating functions.

One can verify easily that the following relations hold:
\begin{eqnarray}\label{eq:netw-var}
D^\ast(x,y)&=&y+S^\ast(x,y)+P^\ast(x,y),\nonumber \\
\tD(x,y) &=& y + \tS(x,y) + \tP(x,y), \\
\hD(x,y) &=& y + \hS(x,y) + \hP(x,y).\nonumber
\end{eqnarray}
Recall that a series network can be decomposed into a network that is not series and an arbitrary SP network. Similarly to Equations~\eqref{eq:S1} and~\eqref{eq:S2}, we get the following equations for the generating functions associated with the networks from the classes $\cS$, $\ctS$, and $\chS$.

\begin{eqnarray}\label{eq:ser-var}
S^\ast(x,y)   &=& \big(D^\ast(x,y)-S^\ast(x,y)\big)xD^\ast(x,y), \nonumber \\
\tS(x,y) &=& \big(\tD(x,y) - \tS(x,y)\big)x D^\ast(x,y) + \big(D^\ast(x,y) -S^\ast(x,y)\big) x \tD(x,y),\\
\hS(x,y) &=& \big(D^\ast(x,y) -S^\ast(x,y)\big)x\hD(x,y) + \big(\hD(x,y)-\hS(x,y)\big)x D^\ast(x,y)\nonumber\\
         & &  +2\big(\tD(x,y)-\tS(x,y)\big)x\tD(x,y).\nonumber
\end{eqnarray}

Finally, a parallel network can be decomposed into a set of at least one series network if the root edge is present, or into at least two series networks, otherwise. In the first case we need to distinguish whether the root edge is in a distinguished spanning tree or not. By a careful case distinction, we get the following three equations:
\begin{eqnarray}\label{eq:par-var}
P(x,y)   &=& S(x,y) \big(\exp(\hS(x,y))-1\big) + (1+y)\tS(x,y)^2\exp(\hS(x,y)) + yS(x,y) \exp(\hS(x,y))\nonumber \\
         & &+2y\tS(x,y)\exp(\hS(x,y)) + y\big(\exp(\hS(x,y) -1\big),\nonumber\\
\tP(x,y) &=& \big(y+ \tS(x,y)\big)\big(\exp(\hS(x,y))-1\big) +y\tS(x,y)\exp(\hS(x,y)), \\
\hP(x,y) &=& \big(\exp(\hS(x,y)) - \hS(x,y) -1\big) +y\big(\exp(\hS(x,y))-1\big).\nonumber
\end{eqnarray}

Equations \eqref{eq:netw-var}-\eqref{eq:par-var} define a system of functional equations that can be analysed by means of Theorem~\ref{thm:drmota} by setting $y=1$.
More precisely, each equation in this system is defined by an analytic function, because the exponential function is an entire function. In addition, easy lower
and upper bounds imply that the radius of convergence of $D(x,1),\, \tD(x,1),\, \hD(x,1)$  is in $(0, R]$, where
$R \approx 0.05668$ is the constant obtained in Lemma~\ref{lem:main_D}.
Solving the system of equations stated in Theorem~\ref{thm:drmota} using \texttt{Maple} we get that $R_2 \approx 0.02407$. Furthermore, Theorem~\ref{thm:drmota} assures that $R_2$ is the singularity of all generating functions appearing in Equations~\eqref{eq:netw-var}-\eqref{eq:par-var} and that they have a square-root expansion in a domain dented at $R_2$. As all generating functions written so far are aperiodic, the singularity is unique on the circle $|x|=R_2$.
\end{proof}

Let $B^\ast(x,y)$ denote the EGF associated with the class of all 2-connected SP graphs carrying two spanning trees.
As in Lemmas~\ref{lem:dissymmetryB} and~\ref{lem:singB} one can write $B^\ast(x,y)$ as an analytic combination of the generating functions from Lemma~\ref{lem:varD} by using Tutte's decomposition and Equation~\eqref{eq:dissymmetry}. Therefore, the dominant singularity of $B^\ast(x,1)$ is the same as the coinciding singularity of the EGFs appearing in Lemma~\ref{lem:varD}, which is $R_2$, and moreover, $B^\ast(x,1)$ has a square-root expansion in a domain dented at $R_2$. By Theorem~\ref{thm:transfer} we get that the number of 2-connected SP graphs on $n$ vertices carrying two spanning trees is asymptotically equal to
$\Theta\left(n^{-5/2}R_2^{-n}n!\right)$.

Let $\mathcal B_n$ denote the number of 2-connected SP graphs on $n$ vertices and $\mathcal{SB}_n^\ast$ the number of 2-connected SP graphs on $n$ vertices carrying two spanning trees. The second moment of $Z_n$ can be calculated in the following way:
$$\mathbb E[Z_n^2]= \sum_{G\in \mathcal B_n} s(G)^2 \mathbb P[G] = \frac{|\mathcal{SB}_n^\ast|}{|\mathcal{B}_n|} = \frac{[x^n]B^\ast(x,1)}{|\mathcal  B_n|},$$
where again $s(G)$ denotes the number of spanning trees in a graph $G$. Recall that the number of 2-connected SP graphs on $n$ vertices is asymptotically equal to $\Theta\left(n^{-5/2} r^{-n}n!\right)$, where $r\approx 0.12800$. This means that the second moment of $Z_n$ is asymptotically equal to $\Theta\left(\varpi_2^{-n}\right)$, where $\varpi_2^{-1} \approx 5.31718$. The same holds for $\mathrm{Var}(Z_n) = \mathbb E[Z_n^2] - \mathbb E[Z_n]^2$ since $\mathbb E[Z_n]$ is approximately equal to $\Theta\left(\varrho^{-n}\right)$, where $\varrho^{-1} \approx 2.08415$ by Theorem~\ref{thm:main}. In particular, this implies that $Z_n$ is not concentrated around its expected value.




\section{Spanning trees in $2$-trees} \label{sec:2trees}

In Subsection~\ref{sub:fix-y} we analysed the exponential growth constant of the expected number of spanning trees in a random 2-connected SP graph with respect to its edge density. Observe that if the edge density of SP graphs tends to its maximum, we reach the class of 2-trees. In this section we present an alternative, direct way to compute this growth constant in the setting of 2-trees. More precisely, we give a proof of Theorem~\ref{thm:2trees}, which is again based on the Symbolic Method, the extension of the Dissymmetry Theorem to tree-decomposable classes (Theorem~\ref{thm:dissymmetry}), and the singularity analysis of generating functions. As in the previous section, we use $x$ and $y$ to mark vertices and edges, respectively. In this particular scenario the number of edges is determined once fixing the number of vertices. However, for pedagogical reasons, we will write all the equations keeping track of both parameters. Recall that all graphs in this paper are considered to be labelled, unless otherwise specified.

In order to obtain the expected number of spanning trees in a random 2-tree, we aim to determine the asymptotic precise estimate of the number of 2-trees as well as of the number of 2-trees carrying a distinguished spanning tree. As a first step, Lemma~\ref{lem:labelled2trees} provides the singularity $\varrho_T$ and the singular expansion in a domain dented at $\varrho_T$ of the generating function $T(x,y)$ associated with the class of 2-trees.

\begin{lemma}
\label{lem:labelled2trees}
For $y=1$ it holds that $T(x,y)$ has a unique square-root type singularity $\varrho_{T} = 1/(2e)$ and admits the following singular expansion in a domain dented at $x=\varrho_T$:
$$T(x,1) = \frac{1}{12}e^{-3/2}-\frac{3}{16}e^{-3/2}X^2+\frac{\sqrt{2}}{48}e^{-3/2}X^3+ \bigO(X^4),$$
where $X= \sqrt{1-x/\varrho_T}$.
\end{lemma}

\begin{proof}
Let $\cbT$ denote the class of labelled 2-trees rooted at an edge the endpoints of which do not bear a label and let $\overline{T}(x,y)$ denote its associated generating function.
By the rules of the Symbolic Method for pointing operations we have the following relation between $T(x,y)$ and $\overline T(x,y)$:
$$y\frac{\partial}{\partial y} T(x,y) = \frac{x^2}{2} \overline{T}(x,y).$$
Observe that we had to add the factor $x^2/2$ on the right-hand side because we had to add labels to the endpoints of the root edge.
By integrating by substitution we get that
$$T(x,y) = \frac{x^2}{2} \int_{0}^{y} \frac{\overline{T}(x,z)}{z} dz = \frac{x^2}{2} \left(\overline{T}(x,y) - \frac{2}{3}x\overline{T}(x,y)^3\right).$$

In order to compute the radius of convergence of $T(x,y)$, it suffices to compute the radius of convergence of $\overline{T}(x,y)$ since their values coincide by the latter equation. A graph in $\cbT$ can be reconstructed by merging at the root edge a set of pairs of graphs in $\cbT$ that share a vertex, see Figure~\ref{fig:2tree}. Using the Symbolic Method this gives rise to the following equation for $\overline{T}(x,y)$:
$$\overline{T}(x,y) = y \exp\big(x\overline{T}(x,y)^2\big).$$
\begin{figure}[htb]
\begin{center}
\includegraphics[width=5.5 cm, page=6, trim= 0.5cm 0 0 0]{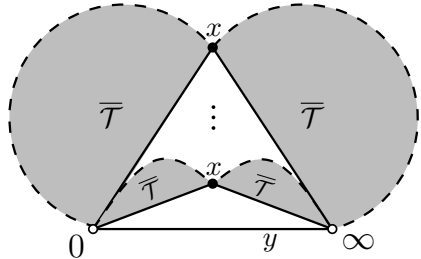}
\caption{Decomposition of rooted 2-trees.}
\label{fig:2tree}
\end{center}
\end{figure}
By Lagrange's Theorem~(see e.g.~\cite{flajolet2009analytic}) we get for every $n, m \geq 0 $ that
$$[x^n][y^m] \overline{T}(x,y) = [x^n] \frac{1}{m} [u^{m-1}] e^{mxu^2} = [x^n] \frac{1}{m} [u^{m-1}] m^{\frac{m-1}{2}} x^{\frac{m-1}{2}} = \left\{
        \begin{array}{ll}
            \frac{1}{\left(\frac{m-1}{2}\right)!}m^{\frac{m-3}{2}} &  \text{ if } n = \frac{m-1}{2} \\
            0 &  \text{ otherwise}.
        \end{array}
    \right. $$

We study the case $y=1$ since the number of edges of a 2-tree is determined by its number of vertices. The inverse function of $T(x,1)$ is given by $\psi(u)= \log(u)/u^2.$ Let $\tau >0$ such that $\psi'(\tau) = \big(1-2\log(\tau)\big)/\tau^3= 0$, implying $\tau = \exp(1/2)$. Then we know from the Inverse Function Theorem that we have for the singularity $\varrho_{T}$ of $\overline{T}(x,1)$  that $\varrho_{T} =\psi (\tau) = 1/(2e)$. Therefore, $\overline T(x,1)$ and hence also $T(x,1)$ have a square-root type singular expansion in a domain dented at $\varrho_T$. By the method of indeterminate coefficients we get that the singular expansion of $T(x,1)$ is of the form as stated in the lemma.
Finally, by aperiodicity of the generating functions under study, the dominant singularity of $T(x,1)$ is unique.
\end{proof}

Now let us turn to 2-trees carrying a spanning tree. We denote this class by $\cT^s$ and its associated generating function by $T^s(x,y)$. Similarly to Section~\ref{sec:trees}, we need to first analyse edge-maximal SP networks to get access to the singular behaviour of $T^s(x,y)$. For this purpose, let $D^{\overline{r}}(x,y)$ denote the EGF associated with the set of all edge-maximal SP networks carrying a spanning tree that contains the root edge. Furthermore, let $D^{r}(x,y)$ denote the EGF associated with the set of all edge-maximal SP networks carrying a spanning tree that does not contain the root edge. Finally, let $D^\circ(x,y)$ denote the EGF associated with the set of all edge-maximal SP networks with a distinguished spanning forest that consists of exactly two components each of which contains one of the poles. Observe that we have $D^{\overline{r}}(x,y) = D^\circ(x,y)$.
Lemma~\ref{lem:maximalSPnetworks} gives us the singular behaviours of $D^{\overline r}=D^\circ$ and $D^r$ for $y=1$.

\begin{lemma}
\label{lem:maximalSPnetworks}
We have $D^{\overline r} = D^\circ$ and for $y=1$ the generating functions $D^{\overline r}$ and $D^{r}$ have the same unique square-root singularity $R_T \approx 0.07197$. Furthermore, the singular expansions (with rounded coefficients) of $D^{\overline r}(x,1)$ and $D^r(x,1)$ in a domain dented at $x =R_T $ are:
\begin{align*}
D^{\overline r}(x,1)&= 1.46516 - 0.77028 X+ 0.53282 X^2- 0.40927 X^3 + \bigO(X^4), \\
D^r(x,1) &= 0.34588 -0.77028 X+ 0.87870 X^2- 0.92279 X^3 + \bigO(X^4),
\end{align*}
where $X= \sqrt{1-x/R_T}$.
\end{lemma}

\begin{proof}
Using the Symbolic Method (see also Figure~\ref{fig:SPnetworks}) one can easily verify that the following system of equations hold:
\begin{equation}
\label{eq:maximalSPnetworks}
\begin{aligned}
D^{\overline r}(x,y) &= y \exp\Big(2\big(D^{\overline{r}}(x,y) +D^{r}(x,y)\big)xD^{\overline r}(x,y)\Big),\\
D^{r}(x,y) &= yx \big(D^{\overline r}(x,y)+D^{r}(x,y)\big)^2\exp\Big(2\big(D^{\overline r}(x,y)+D^{r}(x,y)\big)xD^{\overline r}(x,y)\Big).
\end{aligned}
\end{equation}
\begin{figure}[htb]
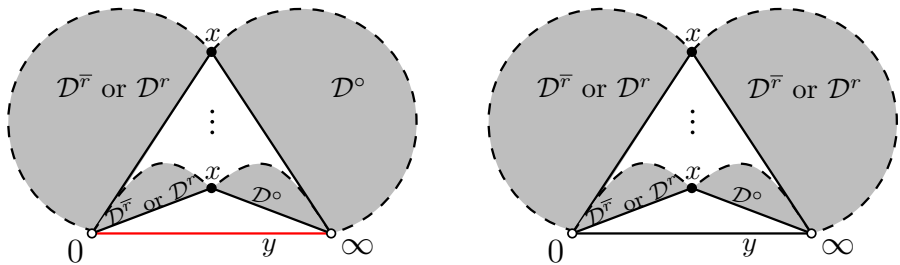

\begin{center}
\includegraphics[width=5.5 cm, trim= 0.5cm 0 0 0, page=4]{Rbrick.pdf}
\quad \quad
\includegraphics[width=5.5 cm, trim= 0.5cm 0 0 0, page=5]{Rbrick.pdf}
\caption{Possible decompositions of an edge-maximal SP network with a distinguished spanning tree that contains the root edge (depicted on the left side)/does not contain the root edge (depicted on the right side).}
\label{fig:SPnetworks}
\end{center}
\end{figure}

We may set $y=1$ since the number of edges of 2-trees is always given by the number of vertices. The two equations in~\eqref{eq:maximalSPnetworks} are defined by entire functions. As these structures carry spanning trees, the corresponding singularity is smaller than $\varrho_T$, and hence their singularity $R_T$ is finite.
We can apply Theorem~\ref{thm:drmota} and know therefore that $D^{\overline r}(x,1)$ and $D^{r}(x,1)$ have the same singularity $R_T$ and that they have a square-root singular expansion in a domain dented at this singularity.
Solving the system of equations stated in this theorem with \texttt{Maple} yields that $R_T \approx 0.07197$.
By aperiodicity of the counting formulas, this is the unique smallest dominant singularity.
By means of indeterminate coefficients we get the exact coefficients of the singular expansions of $D^{\overline r}(x,1)$ and $D^r(x,1)$ which are the ones as stated in the lemma.
\end{proof}

\begin{lemma}
\label{lem:2trees}
Let $y=1$. Then, $T^s(x,1)$ has a square-root singularity, which is the singularity $R_T$ of $D^{\overline r}(x,1)$ and $D^{r}(x,1)$. Moreover, $T^s(x,1)$ has a singular expansion of the following form in a domain dented at $x=R_T\approx 0.07197$:
$$T^s(x,1) =T^s_0 + T^s_2 X^2 + T^s_3 X^3 +\bigO(X^4),$$
where $X= \sqrt{1-x/R_T}$ and $T^s_0 \approx 0.00290$, $T^s_2 \approx -0.00669$, and $T^s_3\approx 0.00133$ are computable constants.
\end{lemma}

\begin{proof}
We define the \emph{$\De$-tree} $\overline\tau(G)$ of a graph $G$ as the bipartite graph describing the incidences between edges and triangles of $G$. More precisely, the node set of $\overline\tau(G)$ is given by $E(G) \cup \big\{\{x_1,x_2,x_3\}: \{x_i, x_j\} \in E(G) \text{ for each } i\neq j\in \{1,2,3\}\big\}$ and two nodes $e$ and $\Delta$ are neighbours if and only if $e\subsetneq \Delta$ in $G$.

Since we are dealing with graphs carrying a spanning tree we need to encode the information on the distinguished spanning tree also in the associated $\De$-tree. For this reason, we define five different types of nodes of $\overline\tau(G)$ if $G \in \cTs$. Let $V_{\overline{e}}$ and $V_{e}$ denote the set of vertices of $\overline\tau(G)$ associated with edges of $G$ that are, respectively are not contained in the distinguished spanning tree of $G$.  Moreover, for $i\in \{1,2,3\}$, let $V_{\Delta_i}$ denote the sets of vertices of $\overline\tau(G)$ associated with triangles of $G$ which contain exactly $3-i$ edges that are in the distinguished spanning tree.

Observe that for every $G\in \cTs$ the connected $\De$-tree $\overline\tau(G)$ associated with it is uniquely defined. We claim that $\overline\tau(G)$ is a tree. Indeed, assume for a contradiction that there exists a cycle $C= \Delta^{1}e^{1}\Delta^{2}\ldots  \Delta^{k}e^{k}\Delta^{1}$ in $\overline\tau(G)$. Let $G'$ be the induced subgraph of $G$ on the vertex set $\bigcup_{i\in[k]} \Delta^{i}$. In particular, being edge-maximal and $K_4$-minor free, $G'$ is also a 2-tree. However, $G'$ does not contain a vertex of degree 2, a contradiction. As a consequence, $\overline\tau(G)$ is indeed a tree. This means that $\cTs$ is a tree-decomposable class and therefore we can apply Theorem~\ref{thm:dissymmetry}. Since $\De$-trees are bipartite, the equation in Theorem~\ref{thm:dissymmetry} simplifies to $\cTs \simeq \cTs_{\circ} - \cTs_{\circ - \circ}$. The class $\cTs_{\circ}$ is naturally partitioned into the five classes $\cTs_{\overline e}$, $\cTs_{e}$, $\cTs_{\Delta_1}$, $\cTs_{\Delta_2}$, and $\cTs_{\Delta_3}$ depending on whether the distinguished node is from the set $V_{\overline e}$, $V_e$, $V_{\Delta_1}$, $V_{\Delta_2}$, or $V_{\Delta_3}$. The class $\cTs_{\circ-\circ}$ is partitioned into the classes $\cTs_{\overline e - \Delta_1}$, $\cTs_{\overline e - \Delta_2}$, $\cTs_{e-\Delta_1}$, $\cTs_{e-\Delta_2}$, and $\cTs_{e-\Delta_3}$ by the structure of $\overline\tau(G)$. In terms of their associated generating functions these facts translate into
\begin{equation}
\label{eq:Ts}
T^s = T^s_{\overline e} + T^s_{e} + T^s_{\Delta_1} + T^s_{\Delta_2} + T^s_{\Delta_3} - T^s_{\overline e - \Delta_1} -T^s_{\overline e - \Delta_2} -T^s_{e- \Delta_1}-T_{e-\Delta_2}-T^s_{e-\Delta_3}.
\end{equation}
 Using the Symbolic Method, it is not difficult to check that the following equations are true. Recall that we have $D^{\overline r}(x,y) = D^\circ(x,y)$.

$$T^s_{\overline e}(x,y) = \frac{x^2}{2} D^{\overline r}(x,y), \qquad T_{e}^s= \frac{x^2}{2} D^r(x,y), \qquad T_{\Delta_1}^s(x,y) = \frac{x^3}{2} \left(D^{\overline r}(x,y)\right)^3,$$
$$T^s_{\Delta_2}(x,y)  = x^3 \left(D^{\overline r}(x,y)\right)^2 D^r(x,y), \qquad
T^s_{\Delta_3}(x,y)  = \frac{x^3}{2} D^{\overline r}(x,y) \left(D^r(x,y)\right)^2$$
and
$$T^s_{\overline e - \Delta_1}(x,y) = x^3\left(D^{\overline r}(x,y)\right)^3 , \qquad T_{\overline e - \Delta_2}^s(x,y) = x^3\left(D^{\overline r}(x,y)\right)^2D^r(x,y),$$
$$T^s_{e- \Delta_1}(x,y) = \frac{1}{2}x^3\left(D^{\overline r}(x,y)\right)^3, \qquad T_{e-\Delta_2}^s(x,y) = 2 x^3 \left(D^{\overline r}(x,y)\right)^2 D^r(x,y),$$
$$T^s_{e-\Delta_3}(x,y) = \frac{3}{2} x^3 D^{\overline r}(x,y) \left(D^{r}(x,y)\right)^2.$$

In particular, $T^s(x,y)$ can be expressed in terms of $x$, $D^{\overline r}(x,y)$, and $D^{r}(x,y)$ by plugging the previous equations in Equation~\eqref{eq:Ts}. One can therefore easily verify that the dominant singularity of $T^s(x,1)$ is the same as the coinciding one of $D^{\overline r}(x,1)$ and $D^r(x,y)$, namely $R_T$.  Finally, we obtain the singular expansion of $T^s(x,1)$ in a domain dented at $R_T$ by using Equation~\eqref{eq:Ts} and the singular expansions of $D^{\overline r}(x,1)$ and $D^r(x,1)$ from Lemma~\ref{lem:maximalSPnetworks}.
\end{proof}

Now we have all ingredients that we need to finally prove Theorem~\ref{thm:2trees}.

\begin{proof}[Proof of Theorem~\ref{thm:2trees}]
Let $U_n$ denote the number of spanning trees in a random 2-tree on $n$ vertices. Then, it holds that
\begin{equation}
\label{eq:2trees}
\mathbb{E}[U_n]=\sum_{G\in \mathcal{T}_n} s(G) \mathbb{P}[G]=\frac{\sum_{G\in \mathcal{T}_n}s(G)}{|\mathcal{T}_n|}=\frac{|\mathcal{T}^s_n|}{|\mathcal{T}_n|}=\frac{[x^n]T^s(x,1)}{|\mathcal{T}_n|},
\end{equation}
where $\cT_n$ and $\cT^s_n$ denote the set of 2-trees on $n$ vertices and the set of 2-trees on $n$ vertices carrying a distinguished spanning tree, respectively, and $s(G)$ denotes the number of spanning trees in a graph $G$.
By Lemma~\ref{lem:2trees} and Theorem~\ref{thm:transfer} we get that the number of 2-trees on $n$ vertices that carry a distinguished spanning tree is asymptotically equal to $\frac{T^s_3}{\Gamma(-3/2)}n^{-5/2} R_T^{-n} n!$.
Furthermore, it follows from Lemma~\ref{lem:labelled2trees} and Theorem~\ref{thm:transfer} that the number of 2-trees on $n$ vertices is asymptotically equal to $\frac{\sqrt{2} e^{-3/2}}{48 \Gamma(-3/2)}n^{-5/2} \varrho_T^{-n} n!$. Dividing the former by the latter as in Equation~\eqref{eq:2trees}, we obtain that the expected value of $U_n$ is asymptotically equal to $s_2 \varrho_2^{-n}$, where $s_2\approx 0.14307$ and $\varrho_2^{-1} \approx 2.55561$.
\end{proof}




\section{Spanning trees in series-parallel graphs with fixed excess}
\label{sec:fixedexcess}

Recall that the excess of a graph $G$ is the number of its edges minus the number of its vertices. In this section we address the following question: given a fixed integer $k$ (not depending on $n$), what is the expected number of spanning trees in a random connected SP graph with $n$ vertices and excess equal to $k$ if $n$ is large? In order to study this question we exploit the structure of graphs with fixed excess introduced by Wright \cite{Wright1977, Wright1978}. The structure of graphs with fixed excess has been applied successfully in a wide variety of situations (see e.g.~\cite{Bernardi2012, Chapuy2010,NoyRamos2014,Rue2013}).

We write $\overline{C}_k(x)$ for the EGF associated with connected SP graphs with excess equal to $k$ and with a distinguished spanning tree. We denote by $C_k(x)$ the EGF of connected SP graphs with fixed excess equal to $k$. We also use the EGF for rooted labelled trees, that we denote by $W(x)$. We recall that $W(x)$ satisfies the functional equation $W(x)=x \exp(W(x))$. Concerning the singular expansion, $W(x)$ has a unique square-root type singularity at $x=e^{-1}$, and in a domain dented at this point $W(x)$ has an expansion of the form
\begin{equation*}
W(x)= 1- \sqrt{2}Z+ \bigO(Z^2),
\end{equation*}
where $Z=\sqrt{1-ex}$.

Let $G$ be a graph with excess $k$.
Starting from $G$, we build a multigraph that we call the \emph{kernel} of $G$ in the following way:
we start deleting recursively vertices of degree one, thus obtaining the \emph{core} of the graph (see for instance~\cite{NoyRamos2014}).
Then we continue by dissolving vertices of degree two, i.e.~replacing the two edges incident to a vertex of degree two by a single edge.
The resulting connected multigraph $K(G)$ (the \emph{kernel} of $G$) has minimum degree greater or equal than $3$ and fixed excess equal to $k$.
The vertices of the kernel of $G$ can be also labelled in the following way: the surviving  $|K(G)|$ vertices in $K(G)$ carry labels from $1$ to $n$ (if $|G|=n$). The labels induce then an order of the vertices.
The final labels in $[1,|K(G)|]$ of the kernel are the ones induced by these ordering of the primitive labels.

If $G$ has a distinguished spanning tree, the construction of the kernel of $G$ induces also a spanning tree in $K(G)$.
Indeed, observe that the edges in $G$ belonging to this distinguished structure induce a spanning tree in its core.
Then, the distinguished spanning tree of the kernel is induced by the spanning tree in the corresponding core.
In particular, possible loops do not belong to the induced spanning tree of the kernel.
Additionally, if there is a multiple edge, at most one of the copies belongs to the spanning tree.
See an example of the construction of the kernel in Figure~\ref{fig:fixed-excess} together with the distinguished spanning structure.
\begin{figure}[htb]
\begin{center}
\includegraphics[width=14.8 cm, height=3.5cm]{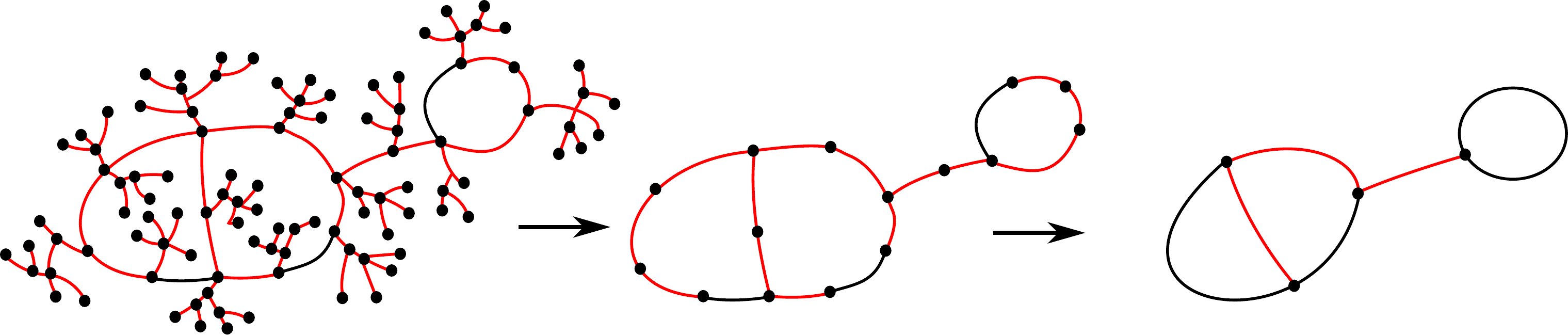}
\caption{Construction of the kernel of a graph.}
\label{fig:fixed-excess}
\end{center}
\end{figure}

It is straightforward how to reverse this construction: starting from a given kernel $C$ we build directly the initial graph by pasting over each vertex of $C$ a rooted labelled tree and substituting each edge by a sequence of rooted labelled trees.
In our approach, dealing with spanning trees, we need to consider a slight modification for edges not belonging to the spanning tree.
More precisely, we need to substitute in this case each edge (neither a loop nor a multiple edge) by
\begin{equation*}
1+2W(x)+3W(x)^2+\dots= \sum_{r\geq 0} (r+1) W(x)^r=\frac{1}{(1-W(x))^2},
\end{equation*}
whose singular expansion in a domain dented at $x=e^{-1}$ is of the form $\frac{1}{2} Z^{-2}+\bigO(Z^{-1})$, where again $Z=\sqrt{1-ex}$.
The reason for this is that whenever we paste a sequence of rooted trees over the edge under consideration we must maintain the spanning structure acyclic and connected. If in the sequence we paste $r$ different trees, the edge becomes subdivided into $r+1$ edges.
We need to choose then which one of these $r+1$ edges is not in the spanning tree.

In the case of the loop the previous sequence must start at the term $2W(x)$ in order to obtain at the end a simple graph.
Similarly, for multiple edges we need to assure that at the end we finally obtain a simple graph.
This case-by-case analysis suggests that generating functions would be somehow involved (see the similar problem in general graphs in \cite{Wright1977, Wright1978}).
However, we show that we are able to find closed formulas for the asymptotic estimates by means of a sequence of reductions.

In the problem under study, our graph $G$ is a connected SP graph, hence its kernel $K(G)$ is a SP multigraph, i.e.~a $K_4$-minor free multigraph.
It is also obvious that the kernel is a planar multigraph. Hence, by Euler's relation, the number of possible kernels of excess $k$ is finite.
Observe that for a fixed value of $k$, a multigraph $K$ of excess $k$ maximizes its number of edges if and only if $K$ is cubic (see the same argument in Lemma 4.3.1 of \cite{Rue2013} for maps on surfaces).
Hence, we can restrict ourselves to the study of connected SP graphs arising from a cubic kernel as these ones will provide the main contribution to the asymptotic.

For technical reasons we will deal with weighted cubic multigraphs: the \emph{weight} (called the \emph{compensation factor} in \cite{Janson1993}, see also \cite{Kang2012}), has the following meaning: when substituting edges of the kernel by sequences of rooted trees, a loop has two possible orientations
that give the same simple graph.
A double (triple) edge can be permuted in two (six) ways and still produce the same object.
The enumerative study of weighted cubic multigraphs will be postponed until the end of this section.
We denote by $g_k$ the number of weighted cubic SP multigraphs with excess $k$. Additionally, we denote by $\overline{g}_k$  the number of weighted cubic SP multigraphs with excess $k$ that carry a distinguished spanning tree.
In the rest of this section, all cubic multigraphs are considered to be weighted. 

Observe that a cubic multigraph of excess $k$ has $2k$ vertices and $3k$ edges. Therefore, any spanning tree has $2k-1$ edges and hence $k+1$ edges are \emph{not} used in the spanning structure. The first lemma gives a lower bound for $[x^n]\overline{C}_k(x)$:

\begin{lemma}\label{lemma:fixedexcess}
For fixed $k>1$, the following inequality holds:
\begin{equation} \label{eq:low-bound}
\overline{g}(k)[x^n] \frac{(3-2W(x))^{k+1} W(x)^{6k}}{(1-W(x))^{4k+1}}   < [x^n] \overline{C}_k(x).
\end{equation}
\end{lemma}

\begin{proof}
The proof is reminiscent to the proof of Lemma 3 in~\cite{Noy2015}. The EGF on the left hand side of Equation~\eqref{eq:low-bound} can be written in the following way:
$$W(x)^{2k}\frac{W(x)^{2k-1}}{(1-W(x))^{2k-1}} \cdot \frac{(3-2W(x))^{k+1}W(x)^{2k+2}}{(1-W(x))^{2k+2}}.$$
This can be interpreted as follows: for a given cubic SP multigraph with a distinguished spanning tree we
\begin{enumerate}
\item[(a)] paste a rooted labelled tree over each of the $2k$ vertices,
\item[(b)] paste a sequence of at least one rooted labelled trees over each of the $2k-1$ edges belonging to the spanning tree of the kernel.
\item[(c)] paste a sequence of at least one rooted labelled trees over each of the $2k+2$ edges not belonging to the spanning tree of the kernel, and then decide which of the new edges does not belong to the resulting spanning tree.
\end{enumerate}
Points~(a) and~(b) contribute with terms $W(x)^{2k}$ and $(W(x)/(1-W(x)))^{2k-1}$, respectively, where the second term is associated with a sequences of at least one rooted labelled tree.
For the computation of Point~(c), recall that loops are not in spanning trees.
Hence, to guarantee that the final graph is simple we need sequences of length at
least two for edges which do not belong to the spanning tree (length one is enough for multiple edges, but length two is needed for loops).
Hence, the computation in point (c) arise from the fact that
$$\sum_{r\geq 2}(r+1)W(x)^{r}=\frac{(3-2W(x))W(x)^2}{(1-W(x))^{2}}.$$
This construction is injective and does \emph{not} give all possible connected SP graphs with excess $k$. Summing all over all possible weighted cubic SP multigraphs we obtain the result as claimed.
\end{proof}

The next step now is to get an upper bound for $[x^n]\overline{C}_k(x)$. This is provided by the following lemma:

\begin{lemma}\label{lemma:fixedexcess2}
For fixed $k>1$, the following inequality holds:
\begin{equation} \label{eq:upper-bound}
[x^n] \overline{C}_k(x) < \overline{g}(k)\left([x^n] \frac{W(x)^{2k}}{(1-W(x))^{4k+1}}\right) (1+o(1)).
\end{equation}
\end{lemma}

\begin{proof}
The statement is proved by applying a similar argument to the one of Lemma \ref{lemma:fixedexcess}.
For a fixed cubic multigraph with a distinguished spanning tree, we now paste over each edge an arbitrary sequence of rooted trees (possibly empty), and taking care of the special requirement on edges of the kernel such that do not belong to the initial spanning tree.
In this way we generate \emph{all} SP graphs carrying a spanning tree with excess $k$.
Observe that it is possible to generate graphs which are not simple.
Hence, this construction only gives an upper bound.
Finally, the term $o(1)$ in Equation~\eqref{eq:upper-bound} arises from the set of kernels that are not cubic.
\end{proof}

We can now combine both lemmas to get an asymptotic estimate for $[x_n] \overline{C}_k(x)$:

\begin{proposition}\label{prop:estimate}
For fixed $k>1$, the following asymptotic estimate in $n$ holds:

\begin{equation} \label{eq:estimate}
[x^n] \overline{C}_k(x) = \overline{g}_k  \frac{1}{2^{2k+1/2}} \frac{n^{2k-1/2}}{\Gamma(2k+1/2)}  e^n (1+o(1)).
\end{equation}
\end{proposition}

\begin{proof} We apply the Transfer Theorem for singularity analysis (Theorem~\ref{thm:transfer}) to Equations ~\eqref{eq:low-bound} and~\eqref{eq:upper-bound}. In particular we get that

$$[x^n] \overline{C}_k(x) >\overline{g}_k[x^n] \frac{(3-2W(x))^{k+1} W(x)^{6k}}{(1-W(x))^{4k+1}} = \overline{g}_k \frac{1}{2^{2k+1/2}} \frac{n^{2k-1/2}}{\Gamma(2k+1/2)} e^n  (1+o(1)),$$
and concerning the upper bound,
$$[x^n] \overline{C}_k(x) < \overline{g}_k\left([x^n] \frac{W(x)^{2k}}{(1-W(x))^{4k+1}}\right) (1+o(1))= \overline{g}_k \frac{1}{2^{2k+1/2}} \frac{n^{2k-1/2}}{\Gamma(2k+1/2)} e^n (1+o(1)).$$

In both estimates we have exploited the fact that $W(e)=1$.
As these estimates have the same singular behaviour we conclude the estimate in Equation~\eqref{eq:estimate}.
\end{proof}

Before moving to the computation of $\overline{g}_k$ and $g_k$, let us mention that similar arguments give estimates for $[x^n]C_k(x)$. Indeed, using the argument of Lemma 3 in \cite{Noy2015} (or \emph{mutatis mutandis} the previous arguments for $\overline{C}_k(x)$) one gets the following upper and lower bounds for $[x^n]C_k(x)$:
$$ g_k [x^n] \frac{W(x)^{8k}}{(1-W(x))^{3k}} <[x^n] C_k(x) < g_k\left([x^n] \frac{W(x)^{2k}}{(1-W(x)
)^{3k}}\right) (1+o(1)).$$
Again, by applying the Transfer Theorem for singularity analysis (Theorem~\ref{thm:transfer}) we get the estimate
\begin{equation}\label{eq:estimate2}
[x^n]C_k(x)= g_k\frac{1}{2^{3k/2}} \frac{n^{3k/2-1}}{\Gamma(3k/2)} e^n  (1+o(1)).
\end{equation}

\subsection*{Enumeration of cubic cores with spanning structures and proof of Theorem~\ref{thm:excess}} \label{sec: cubic-cores}

We complete the picture by getting asymptotic formulas (in $k$) for both $g_k$ and $\overline{g}_k$.
In this section we get asymptotic estimates when $k$ is large enough.
The main results we implicitly use are the Transfer Theorems, joint with the fact that singular expansions can be integrated on dented domains (see for instance Theorem VI.9 in \cite{flajolet2009analytic}).
Additionally, given a value $k$, one can obtain the corresponding values by getting the Taylor expansions of the generating functions that will be introduced in the next lines.

Let us start studying $g_k$. As mentioned before in this section, $g_k$ is the number of connected cubic SP multigraphs with weights and excess equal to $k$. The weights are defined rigourously in the following way: given a multigraph with $l_1$ loops, $l_2$ double edges and $l_3$ triple edges, its \emph{weight} is $2^{-l_1-l_2} 6^{-l_3}$. Weights are needed to encode edge symmetries of multigraphs.
Let $G(u)$ be the EGF of connected weighted cubic SP multigraphs, where the variable $u$ marks the excess. $G(u)$ satisfies the following system of functional equations:
\begin{eqnarray}\label{eq:system-planar}
6u\frac{d G(u)}{du}&=&  d(u)+c(u),\nonumber\\
b(u)                 &=& \frac{u}{2}(d(u)+c(u))+\frac{u}{2},\\
c(u)                 &=& s(u)+p(u)+b(u), \nonumber\\
d(u)                 &=& \frac{b(u)^2}{u},\nonumber\\
s(u)                 &=& c(u)^2-c(u)s(u),\nonumber\\
p(u)                 &=& uc(u)+\frac{1}{2}uc(u)^2+\frac{u}{2}.\nonumber
\end{eqnarray}

Full details concerning these equations can be found in Section 3 of~\cite{Noy2015}, building on results on~\cite{Kang2012} (see also \cite{BodirskyKang2007}). The meaning of these EGF is the following: the term  $6u\frac{d G(u)}{du}$ corresponds to the EGF of connected weighted cubic SP multigraphs where an edge (the \emph{root edge}) is marked and oriented (remember that a cubic multigraph of excess $k$ has $3k$ edges and each edge has 2 possible orientations). We have the following cases depending on the nature of the root edge: either the root edge is a loop (term $b(u)$) or a bridge (term $d(u)$) or when deleting it we get a series construction (term $s(u)$) or a parallel construction (term $p(u)$). The term $c(u)$ plays the role of an auxiliary EGF.

We proceed now to analyse this system.

\begin{proposition}\label{prop:asympt_g_k}
The number $g_k= [u^k]G(u)$ of weighted connected cubic SP multigraphs with fixed excess equal to $k$ has the following asymptotic estimate:
\begin{equation*}
g_k = c k^{-5/2} \gamma^{-k} (1+o(1)),
\end{equation*}
where $\gamma= \frac{4}{27}\sqrt{6\sqrt{3}-9} \approx 0.17481$ and $c \approx 0.06034$.
\end{proposition}

\begin{proof}
From the system of equations~\eqref{eq:system-planar} we get a single equation for $c(u)$ by successive elimination. Computations give that $c(u)$ satisfies the algebraic equation
\begin{eqnarray*}
0   &=&     8u+u^2+(-8+24u+6u^2)c(u)+(-4+24u+15u^2)c(u)^2+(8u+20u^2)c(u)^3\\
    &&      +15u^2 c(u)^4+6u^2 c(u)^5+u^2 c(u)^6.
\end{eqnarray*}
We refer to Section VII.8 in \cite{flajolet2009analytic}) for more details. From this equation we deduce that the dominant singularity of $c(u)$ is the smallest positive root of the equation
$$ 19683u^4+7776u^2-256=0,$$
whose value is equal to $\gamma= \frac{4}{27}\sqrt{6\sqrt{3}-9} \approx 0.17481$ ($-\gamma$ is also a root of this polynomial with the same absolute value, but it is easy to check that $c(u)$ is analytic at $u=-\gamma$). Using now Newton's Polygon Method (see Section VII.7 of \cite{flajolet2009analytic}) we get that $c(u)$ has a Puiseux's expansion of the following form in a domain dented at $u=\gamma$:
\begin{equation} \label{eq:expansion_C}
c(u)=c_0+c_1 U+ \bigO(U^2),
\end{equation}
where $U=\sqrt{1-u/\gamma}$, $c_0 \approx 0.61185$ and $c_1 \approx -1.08766$. Using Expansion~\eqref{eq:expansion_C} we deduce that $D(u)$ admits the following singular expansion in a domain dented at $u=\gamma$:
$$d(u)=d_0+d_1 U+ \bigO(U^2) $$
with $d_0 \approx 0.13306$ and $d_1\approx -0.19574$. Finally, using that $6u\frac{d G(u)}{du}=  d(u)+c(u)$ we conclude that the dominant singularity of $G(u)$ is located at $u=\gamma$. The proposition finally follows by integration of the Puiseux's series (by applying Theorem VI.9 from \cite{flajolet2009analytic}) in order to get the singular expansion of $G(u)$ and by the application of the Transfer Theorem (Theorem~\ref{thm:transfer})
\end{proof}

We proceed with the study of $\overline{g}_k$. For this purpose, we refine the system of equations \eqref{eq:system-planar} in the following way: we denote by
$\overline{G}(u)$ the EGF associated with connected weighted cubic SP multigraphs carrying a distinguished spanning tree (as before, $u$ marks the excess).
We study decompositions for  $6u\frac{d \overline{G}(u)}{du}$, which correspond to the EGF of connected weighted cubic SP multigraphs carrying a distinguished spanning tree where an edge is distinguished and oriented.

In the following expressions, we use the subindex $0$ to denote that the distinguished and oriented edge belongs to the spanning tree, while we use the subindex $1$ to denote the opposite. Using this convention, we write $d_0(u)$, $b_0(u)$, $s_0(u)$ and $p_0(u)$ for the cases where this distinguished edge is a bridge, a loop, defines a series construction or a parallel construction, respectively, in such a way that the distinguished edge belongs to the spanning tree of the initial structure. Analogue definitions are done for $d_1(u)$, $b_1(u)$, $s_1(u)$ and $p_1(u)$. $c_0(u)$ and $c_1(u)$ are auxiliary families.

Using the same arguments used to obtain the system of equations~\eqref{eq:system-planar} yields the following, more involved system of equations:
\begin{eqnarray}\label{eq:system-planar2}
6u\frac{d \overline{G}(u)}{du}&=&  d_0(u)+c_0(u) + d_1(u) + c_1(u),\nonumber\\
c_0(u)				&=& b_0(u) + s_0(u) + p_0(u), \\
c_1(u)				&=& b_1(u) + s_1(u) + p_1(u), \nonumber\\
b_0(u)				&=& 0, \nonumber\\
b_1(u)				&=& \frac{u}{2}\big(d_0(u)+c_0(u)\big) + u \big(d_1(u)+c_1(u)\big)+\frac{u}{2}, \nonumber\\
d_0(u)				&=& \frac{b_1(u)^2}{u}, \nonumber\\
d_1(u)				&=& 0, \nonumber\\
s_0(u)				&=& \big(c_1(u)-s_1(u)\big)c_1(u) + \big(c_0(u)-s_0(u)\big)c_1(u) + \big(c_1(u)-s_1(u)\big)c_0(u), \nonumber\\
s_1(u)				&=& \big(c_1(u)-s_1(u)\big)c_1(u), \nonumber\\
p_0(u)				&=& \frac{u}{2}+uc_0(u)+2uc_1(u)+\frac{1}{2}u c_0(u)^2, \nonumber\\
p_1(u)				&=& u + uc_0(u) + 3uc_1(u) +uc_0(u)c_1(u)+2uc_1(u)^2.\nonumber
\end{eqnarray}

We now analyse this system of equations similarly to what we did when studying System~\eqref{eq:system-planar}:

\begin{proposition}\label{prop:asympt_g_k_bar}
The number $\overline{g}_k= [u^k]\overline{G}(u)$ of weighted connected cubic SP multigraphs carrying a spanning tree and with excess equal to $k$ has the following asymptotic estimate:
\begin{equation*}
\overline{g}_k = \overline{c} k^{-5/2} \overline{\gamma}^{-k} (1+o(1)),
\end{equation*}
where $\overline{\gamma} \approx 0.06709$ and $\overline{c} \approx 0.06634$.
\end{proposition}

\begin{proof}
The arguments are exactly the same as in Proposition \ref{prop:asympt_g_k}. From the system of equations~\eqref{eq:system-planar2} we get the following algebraic equation satisfied by $c_0(u)$:
\begin{eqnarray*}
0&=&121u^3+2304u+7656u^2\\
  &&+(51696u-26620u^2-968u^3-4608)c_0(u)\\
  &&+(-384+34532u+30888u^2+3388u^3)c_0(u)^2\\
  &&+(256-1392u-10516u^2-6776u^3)c_0(u)^3\\
  &&+(32-1608u-2376u^2+8470u^3)c_0(u)^4\\
  &&+(-352u-132u^2-6776u^3)c_0(u)^5\\
  &&+(4u+1144u^2+3388u^3)c_0(u)^6\\
  &&+(-44u^2-968u^3)c_0(u)^7+121u^3c_0(u)^8.
\end{eqnarray*}
The singular point now is located at $\overline{\gamma} \approx 0.06709$, and the Puiseux's expansion of $c_0(u)$ around $u=\overline{\gamma}$ is of the form
$$c_0(u)= c_{0,0}+c_{0,1} \overline{W}+\bigO(\overline{W}^2),$$
where $\overline{W}=\sqrt{1-u/\overline{\gamma}}$, $c_{0,0} \approx 0.29896$, and $c_{0,1} \approx -0.47032$. We can then directly obtain the Puiseux's expansion of $c_1(u)$, $b_1(u)$ and $d_1(u)$ from the expansion of $c_0(u)$.  By integrating the equation $6u \frac{d}{du} \overline{G}(u)=c_0(u)+c_1(u)+d_0(u)$ and applying the Transfer Theorem (Theorem~\ref{thm:transfer}) we get the estimate for $\overline{g}_k$ as it is claimed.
\end{proof}

We can now complete the proof of Theorem \ref{thm:excess}.

\begin{proof}[Proof of Theorem ~\ref{thm:excess}]
Due to Proposition \ref{prop:estimate} and Equation \eqref{eq:estimate2}, the value $\mathbb{E}[X_{n,k}]$ is
$$\frac{[x^n] \overline{C}_k(x)}{[x^n]C_k(x)}=\frac{\overline{g}_k \frac{1}{2^{2k+1/2}} \frac{n^{2k-1/2}}{\Gamma(2k+1/2)} e^n}{g_k\frac{1}{2^{3k/2}} \frac{n^{3k/2-1}}{\Gamma(3k/2)} e^n}(1+o(1))=\frac{\overline{g}_k}{g_k} \frac{\Gamma(3k/2)}{\Gamma(2k+1/2)} \left(\frac{n}{2}\right)^{\frac{k+1}{2}}  (1+o(1)).$$

Using now the estimates obtained for $g_k$ and $\overline{g}_k$ in Propositions  \ref{prop:asympt_g_k} and \ref{prop:asympt_g_k_bar} we obtain the result (and the constants) as claimed in the statement of the theorem.\end{proof}

\section{Concluding remarks and open problems}\label{sec:concluding}

In this paper, we have exploited the use of generating functions joint with analytic combinatorics to get exact expressions for the counting generating functions for connected SP graphs (2-connected SP graphs, 2-trees, and SP graphs with fixed excess) with a distinguished spanning tree. As a consequence, we were able to get very precise estimates for the expected number of spanning trees in an object chosen uniformly at random from the family under study. These techniques could be exploited in related families of graphs, as for instance the family of $k$-trees with $k\geq 3$. Even though the analytic techniques used in this paper allow access to higher moments, unfortunately they do not provide a technique to determine the limit law of the number of spanning trees in SP graphs.

With the techniques used in this paper one can also determine, with a bit more effort, the counting generating function of connected SP graphs carrying a distinguished spanning forest with a given number of components. From this one can derive the expected number of components of a random spanning forest in a random connected SP graph. In particular, the main difficulty in this situation is that for encoding networks carrying a distinguished spanning forest, one needs more auxiliary classes than in the restricted case of spanning trees. Roughly speaking, one needs to define classes of SP networks that carry a spanning forest and such that either the two poles are contained in the same component or in different components. 
The analysis of the generating function associated with networks as well as determining and analysing the generating functions associated with 2-connected and connected SP graphs carrying a spanning forest can be done similarly to the case of spanning trees. In the context of planar maps, let us mention that Bousquet-M\'elou and Courtiel~\cite{BousquetMelou15} recently investigated the enumeration of regular planar maps carrying a distinguished spanning forest.

Finally, we would like to discuss similar results on planar graphs.
Following the lines of~\cite{GiNoyRue2013}, the tools developed in the present paper can be extended easily to graphs defined by 3-connected components. When the family under consideration is defined by a finite number of 3-connected graphs, then the techniques used so far can be exploited to get analogue results. This would include, for instance, the family of graphs $\mathrm{Ex}(W_4)$ or $\mathrm{Ex}(W_5)$, where $W_4$ and $W_5$ are the wheel graphs with 4 and 5 external vertices, respectively (see ~\cite{GiNoyRue2013}). In this context, a very interesting question is to extend the results of SP graphs to the random planar graph model. In this situation, the family of 3-connected components is infinite, and one requires extra results arising from map enumeration in order to control counting formulas for T-bricks. Although the number of maps carrying a spanning tree is a well-known fact (see e.g.~\cite{Mullin67}), nothing is known when dealing with 3-connected planar graphs. The problem of getting the expected number of spanning trees in a planar graph chosen uniformly at random will be considered in future investigations.

$$\,$$
\paragraph{\textbf{Acknowledgements}.} 
The authors thank two anonymous referees for their comments that have helped improving a lot the content of the paper and also for pointing out the existing bibliography in the context of maps.
Parts of this work were done when the first author was visiting \emph{FU Berlin}
and when the second author was visiting \emph{TU Hamburg-Harburg}. Both would like to thank the
members of the research groups of Tibor Szab\'o and Anusch Taraz for their warm hospitality.

\bibliographystyle{abbrv} 	
\bibliography{spanningtrees}

\end{document}